\documentclass[11pt,reqno]{amsart}

\newtheorem{proposition}{Proposition}[section]
\newtheorem{lemma}[proposition]{Lemma}
\newtheorem{corollary}[proposition]{Corollary}
\newtheorem{theorem}[proposition]{Theorem}

\theoremstyle{definition}
\newtheorem{definition}[proposition]{Definition}
\newtheorem{example}[proposition]{Example}
\newtheorem{examples}[proposition]{Examples}
\newtheorem{remark}[proposition]{Remark}

\newcommand{\thlabel}[1]{\label{th:#1}}
\newcommand{\thref}[1]{Theorem~\ref{th:#1}}
\newcommand{\selabel}[1]{\label{se:#1}}
\newcommand{\seref}[1]{Section~\ref{se:#1}}
\newcommand{\lelabel}[1]{\label{le:#1}}
\newcommand{\leref}[1]{Lemma~\ref{le:#1}}
\newcommand{\prlabel}[1]{\label{pr:#1}}
\newcommand{\prref}[1]{Proposition~\ref{pr:#1}}
\newcommand{\colabel}[1]{\label{co:#1}}
\newcommand{\coref}[1]{Corollary~\ref{co:#1}}

\newcommand{\exlabel}[1]{\label{ex:#1}}
\newcommand{\exref}[1]{Example~\ref{ex:#1}}
\newcommand{\delabel}[1]{\label{de:#1}}
\newcommand{\deref}[1]{Definition~\ref{de:#1}}
\newcommand{\eqlabel}[1]{\label{eq:#1}}
\newcommand{\equref}[1]{(\ref{eq:#1})}

\def\ot{\otimes}

\newcommand{\Cc}{\mathcal{C}}
\newcommand{\Dd}{\mathcal{D}}

\def\*C{{}^*\hspace*{-1pt}{\Cc}}
\def\text#1{{\rm {\rm #1}}}

\input xy
\xyoption {all} \CompileMatrices

\usepackage{color,amssymb,graphicx}

\begin{document}

\title[Extending structures II: the quantum version]
{Extending structures II: the quantum version}

\author{A. L. Agore}
\address{Faculty of Engineering, Vrije Universiteit Brussel, Pleinlaan 2, B-1050 Brussels, Belgium}
\email{ana.agore@vub.ac.be and ana.agore@gmail.com}

\author{G. Militaru}
\address{Faculty of Mathematics and Computer Science, University of Bucharest, Str.
Academiei 14, RO-010014 Bucharest 1, Romania}
\email{gigel.militaru@fmi.unibuc.ro and gigel.militaru@gmail.com}


\begin{abstract} Let $A$ be a Hopf algebra and $H$ a coalgebra.
We shall describe and classify up to an isomorphism all Hopf
algebras $E$ that factorize through $A$ and $H$: that is $E$ is a
Hopf algebra such that $A$ is a Hopf subalgebra of $E$, $H$ is a
subcoalgebra in $E$ with $1_E \in H$ and the multiplication map
$A\otimes H \to E$ is bijective. The tool we use is a new product,
we call it the unified product, in the construction of which $A$
and $H$ are connected by three coalgebra maps: two actions and a
generalized cocycle. Both the crossed product of an Hopf algebra
acting on an algebra and the bicrossed product of two Hopf
algebras are special cases of the unified product. A Hopf algebra
$E$ factorizes through $A$ and $H$ if and only if $E$ is
isomorphic to a unified product of $A$ and $H$. All such Hopf
algebras $E$ are classified up to an isomorphism that stabilizes
$A$ and $H$ by a Schreier type classification theorem. A coalgebra
version of lazy $1$-cocycles as defined by Bichon and Kassel plays
the key role in the classification theorem.
\end{abstract}

\maketitle

\section*{Introduction}
Let $\Cc$ be a category whose objects are sets endowed with
various algebraic structures $(S)$ and $\Dd$ be a category such
that there exists a forgetful functor $F: \Cc \to \Dd$, i.e. a
functor that forgets some of the structures $(S)$. To illustrate,
the following are forgetful functors:
$$
F: {\mathcal Gr } \to {\mathcal Set}, \quad F: {\mathcal Lie } \to
{\mathcal Vec}, \quad F: {\mathcal Hopf } \to {\mathcal CoAlg},
\quad F: {\mathcal Hopf } \to {\mathcal Alg}
$$
where ${\mathcal Gr }$, ${\mathcal Set}$, ${\mathcal Lie }$,
${\mathcal Vec}$, ${\mathcal Hopf }$, ${\mathcal CoAlg}$,
${\mathcal Alg}$ are the categories of all groups, sets, Lie
algebras, vector spaces, Hopf algebras, coalgebras and
respectively algebras. In this context we formulate a general
problem which may be of interest for many areas of mathematics:

\textbf{Extending Structures Problem (ES):} \textit{Let $F: \Cc
\to \Dd$ be a forgetful functor and consider two objects $C\in
\Cc$, $D \in \Dd$ such that $F(C)$ is a subobject of $D$ in $\Dd$.
Describe and classify all mathematical structures $(S)$ that can
be defined on $D$ such that $D$ becomes an object of $\Cc$ and $C$
is a subobject of $D$ in the category $\Cc$ (the classification is
up to an isomorphism that stabilizes $C$ and a certain type of
fixed quotient $D/C$).}

The classification part of the ES-problem is a challenge for the
introduction of new types of cohomology. The ES-problem
generalizes and unifies two famous and still open problems in the
theory of groups: the \emph{extension problem} of H\"{o}lder
\cite{holder} and the \emph{factorization problem} of Ore
\cite{Ore}. Let us explain this. In \cite{am3} we formulated the
ES-problem at the level of groups, corresponding to the forgetful
functor $F: {\mathcal Gr } \to {\mathcal Set}$: if $A$ is a group
and $E$ a set such that $A \subseteq E$ \cite[Corollary
2.10]{am3}, describe all group structures $(E, \cdot)$ that can be
defined on the set $E$ such that $A$ is a subgroup of $(E,
\cdot)$. In order to do that we have introduced a new product for
groups, called the unified product (\cite[Theorem 2.6]{am3}), such
that both the crossed product (the tool for the extension problem)
and the bicrossed product (the tool for the factorization problem)
of two groups are special cases of it. The unified product for
groups is associated to a group $A$ and a new hidden algebraic
structure $(H, \ast)$, connected by two actions and a generalized
cocycle satisfying some compatibility conditions.

We now take a step forward and formulate the ES-problem at the
level of Hopf algebras corresponding to the forgetful functor $F:
{\mathcal Hopf} \to {\mathcal CoAlg}$:

\textbf{(H-C) Extending Structures Problem:} \textit{ Let $A$ be a
Hopf algebra and $E$ a coalgebra such that $A$ is a subcoalgebra
of $E$. Describe and classify all Hopf algebra structures that can
be defined on $E$ such that $A$ is a Hopf subalgebra of $E$.}

There is of course a dual version of the ES-problem corresponding
to the forgetful functor $F: {\mathcal Hopf} \to {\mathcal Alg}$
to be addressed somewhere else. If at the level of groups the
ES-problem is elementary, for Hopf algebras the problem is more
difficult. Indeed, let $A$ be a group and $E$ a set such that $A
\subseteq E$. For a field $k$ we look at the extension $k [A]
\subseteq k [E]$, where $k[A]$ is the group algebra that is a Hopf
algebra and a subcoalgebra in the group-like coalgebra $k[E]$.
Assume now that $(E, \cdot)$ is a group structure on the set $E$
such that $A$ is a subgroup of $(E, \cdot)$. Thus, we obtain an
extension of Hopf algebras $k [A] \subseteq k [E]$. This extension
of Hopf algebras has a remarkable property: let $H \subseteq E$ be
a system of representatives for the right cosets of the subgroup
$A$ in the group $(E, \cdot)$ such that $1_E \in H$. Since the map
$u : A\times H \rightarrow E$, $u (a, h) = a\cdot h$ is bijective,
we obtain that the multiplication map
$$
k[A] \otimes k[H] \to k[E], \quad a\otimes h \mapsto a\cdot h
$$
is bijective, i.e. the Hopf algebra $k[E]$ \emph{factorizes}
through the Hopf subalgebra $k[A]$ and the subcoalgebra $k[H]$.
This is not valid for arbitrary extensions of Hopf algebras.
Therefore, we have to restrict the (H-C) extending structures
problem to those Hopf algebras $E$ that factorize through a given
Hopf subalgebra $A$ and a given subcoalgebra $H$: we called this
the \textbf{restricted (H-C) ES-problem} and we shall give a
complete answer to it in the present paper. It turns out that $H$
is not only a subcoalgebra of $E$ but will be endowed additionally
with a hidden algebraic structure that will play the role of the
system of representatives for congruence in the theory of groups.

The paper is organized as follows: In the first section we recall
the classical constructions of the crossed product of a Hopf
algebra $H$ acting on an algebra $A$ and of the bicrossed product
(double cross product in Majid's terminology) of two Hopf algebras
$H$ and $A$, as the product that we define will generalize both of
them. In \seref{uni1} we define the concept of an extending
structure of a bialgebra $A$ consisting of a system $\Omega(A) =
\bigl(H, \triangleleft, \, \triangleright, \, f \bigl)$, where $H$
is a coalgebra and an unitary not necessarily associative algebra
such that $A$ and $H$ are connected by three coalgebra maps
$\triangleleft : H \otimes A \rightarrow H$, $\triangleright: H
\otimes A \rightarrow A$, $f: H \otimes H \rightarrow A$
satisfying some natural normalization conditions (\deref{noua1}).
For a bialgebra extending structure $\Omega(A) = \bigl(H,
\triangleleft, \, \triangleright, \, f \bigl)$ of $A$ we define a
product $A \ltimes_{\Omega(A)} H = A \ltimes H$ and call it the
unified product: both the crossed product of an Hopf algebra
acting on an algebra and the bicrossed product (double cross
product in Majid's terminology) of two Hopf algebras are special
cases of the unified product. \thref{1} gives necessary and
sufficient conditions for $A \ltimes H$ to be a bialgebra, which
is precisely the Hopf algebra version of \cite[Theorem 2.6]{am3}
proven for the group case that served as a model for us. The seven
compatibility conditions in \thref{1} are very natural and,
mutatis-mutandis, are the ones (with two reasonable deformations
via the right action $\triangleleft$) that appear in the
construction of the crossed product and the bicrossed product of
two Hopf algebras. \thref{2} proves that a Hopf algebra $E$
factorizes through a Hopf subalgebra $A$ and a subcoalgebra $H$ if
and only if $E$ is isomorphic to a unified product of $A$ and $H$
and gives the answer for the first part of the restricted (H-C)
ES-problem.

\seref{sec3} is devoted to the classification part of the
restricted (H-C) ES-problem. Our view point descends from the
classical classification theorem of Schreier at the level of
groups: all extensions of an abelian group $K$ by a group $Q$ are
classified by the second cohomology group $H^2 (Q, K)$
\cite[Theorem 7.34]{R}. Let $A$ be a Hopf algebra. Two Hopf
algebra extending structures $\Omega(A) = \bigl(H, \triangleleft,
\, \triangleright, \, f \bigl)$ and $\Omega'(A) = \bigl(H,
\triangleleft', \, \triangleright', \, f' \bigl)$ are called
equivalent if there exists $\varphi: A \ltimes' H \rightarrow A
\ltimes H $ a left $A$-module, a right $H$-comodule and a Hopf
algebra map. As in group extension theory we shall prove that any
such morphism $\varphi: A \ltimes' H \rightarrow A \ltimes H $ is
an isomorphism and the following diagram
$$
\xymatrix {& A \ar[r]^{i_{A}} \ar[d]_{Id_{A}} & {A\bowtie H}
\ar[r]^{\pi_{H}}\ar[d]^{\varphi} & H\ar[d]^{Id_{H}}\\
& A\ar[r]^{i_{A}} & {A\bowtie' H}\ar[r]^{\pi_{H}} & H}
$$
is commutative. \thref{3.4} shows that any such morphism $\varphi:
A \ltimes' H \rightarrow A \ltimes H $ is uniquely determined by a
coalgebra lazy $1$-cocycle: i.e.  a unitary coalgebra map $u: H
\rightarrow A$ such that:
$$
h_{(1)} \otimes u(h_{(2)}) = h_{(2)} \otimes u(h_{(1)})
$$
for all $h\in H$. \coref{3.67} is the Schreier type classification
theorem for unified products: the part of the second cohomology
group from the theory of groups will be played now by a special
quotient set $H^{2}_{l, c} (H, A, \triangleleft)$. Also, a
classification result for bicrossed product of two Hopf algebras
is derived from \thref{3.4}.

\section{Preliminaries}
Throughout this paper, $k$ will be a field. Unless specified
otherwise, all algebras, coalgebras, bialgebras, tensor products
and homomorphisms are over $k$. For a coalgebra $C$, we use
Sweedler's $\Sigma$-notation: $\Delta(c) = c_{(1)}\ot c_{(2)}$,
$(I\ot\Delta)\Delta(c) = c_{(1)}\ot c_{(2)}\ot c_{(3)}$, etc
(summation understood). Let $A$ be a bialgebra and $H$ a
coalgebra. $H$ is called a \textit{right $A$-module coalgebra} if
there exists $\triangleleft : H \otimes A \rightarrow H$ a
morphism of coalgebras such that $(H, \triangleleft) $ is a right
$A$-module. For a $k$-linear map $f: H \ot H \to A$ we denote $f
(g, \, h) = f (g\ot h)$; $f$ is the \textit{trivial map} if $f (g,
h) = \varepsilon_H (g) \varepsilon_H (h) 1_A$, for all $g$, $h\in
H$. Similarly, the $k$-linear maps $\triangleleft : H \otimes A
\rightarrow H$, $\triangleright: H \otimes A \rightarrow A$ are
the \textit{trivial actions} if $h \triangleleft a = \varepsilon_A
(a) h$ and respectively $h\triangleright a = \varepsilon_H(h) a$,
for all $a\in A$ and $h\in H$. For further computations, the fact
that $\triangleleft : H \otimes A \rightarrow H$, $\triangleright:
H \otimes A \rightarrow A$ and $f: H \ot H \to A$ are coalgebra
maps can be written explicitly as follows:
\begin{eqnarray}
\Delta_{H}(h \triangleleft a) &{=}& h_{(1)} \triangleleft a_{(1)}
\otimes h_{(2)} \triangleleft a_{(2)}, \,\,\,\,\,
\varepsilon_{A}(h \triangleleft a) =
\varepsilon_{H}(h)\varepsilon_{A}(a)
\eqlabel{6}\\
\Delta_{A}(h \triangleright a) &{=}& h_{(1)} \triangleright
a_{(1)} \otimes h_{(2)} \triangleright a_{(2)}, \,\,\,\,\,
\varepsilon_{A}(h \triangleright a) =
\varepsilon_{H}(h)\varepsilon_{A}(a) \eqlabel{8}\\
\Delta_{A}\bigl(f(g,h)\bigl) &{=}& f(g_{(1)},h_{(1)}) \otimes
f(g_{(2)},h_{(2)}), \,\, \varepsilon_{A}\bigl(f(g,h)\bigl) =
\varepsilon_{H}(g)\varepsilon_{H}(h) \eqlabel{4}
\end{eqnarray}
for all $g, h \in H$, $a \in A$.

\subsection*{Crossed product of Hopf algebras}\selabel{1.1}
The crossed product of a Hopf algebra $H$ acting on a $k$-algebra
$A$ was introduced independently in \cite{BCM} and \cite{DT} as a
generalization of the crossed product of groups acting on
$k$-algebras. Let $H$ be a Hopf algebra, $A$ a $k$-algebra and two
$k$-linear maps $\triangleright: H \otimes A \rightarrow A$, $f: H
\otimes H \rightarrow A$ such that
$$
h \triangleright 1_{A} = \varepsilon_{H}(h)1_{A}, \quad 1_H \rhd a
= a
$$
$$h \rhd (ab) = (h_{(1)} \rhd a) (h_{(2)} \rhd b), \quad f(h,
1_{H}) = f(1_{H}, h) = \varepsilon_{H}(h)1_{A}
$$
for all $h\in H$, $a$, $b\in A$. The \textit{crossed product }$A
\#_{f} \, H$ of $A$ with $H$ is the $k$-module $A\ot H$ with the
multiplication given by
\begin{equation}\eqlabel{001}
(a \# h) (c \# g):= a (h_{(1)}\triangleright c) f\bigl(h_{(2)} ,
g_{(1)}\bigl) \, \# \, h_{(3)}g_{(2)}
\end{equation}
for all $a$, $c\in A$, $h$, $g\in H$, where we denoted $a\ot h$ by
$a\# h$. It can be proved \cite[Lemma 7.1.2]{M} that $A \#_{f} \,
H$ is an associative algebra with identity element $1_A \# 1_H$ if
and only if the following two compatibility conditions hold:
\begin{eqnarray}
[g_{(1)} \triangleright (h_{(1)} \triangleright a)] f\bigl(g_{(2)}
, \, h_{(2)} \bigl) &{=}& f(g_{(1)}, \, h_{(1)}) \bigl ( (g_{(2)}
h_{(2)}) \triangleright a \bigl)\\  \eqlabel{tmc}
\bigl(g_{(1)}
\triangleright f(h_{(1)}, \, l_{(1)})\bigl) f\bigl(g_{(2)}, \,
h_{(2)} l_{(2)} \bigl) &{=}& f(g_{(1)}, \, h_{(1)})
f(g_{(2)}h_{(2)}, \, l)\eqlabel{cc}
\end{eqnarray}
for all $a\in A$, $g$, $h$, $l\in H$. The first compatibility is
called the twisted module condition while \equref{cc} is called
the cocycle condition. The crossed product $A \#_{f} \, H$ was
studied only as an algebra extension of $A$, being an essential
tool in Hopf-Galois extensions theory. If, in addition, we suppose
that $A$ is also a Hopf algebra, a natural question arises:
\emph{when does the crossed product $A \#_{f} \, H$ have a Hopf
algebra structure with the coalgebra structure given by the tensor
product of coalgebras?} In case that $\rhd$ and $f$ are coalgebra
maps, as a consequence of \thref{1}, we will show in
\exref{3exemple} that $A \#_{f} \, H$ is a Hopf algebra if and
only if the following two compatibility conditions hold:
\begin{eqnarray*}
g_{(1)} \otimes g_{(2)} \triangleright a &{=}& g_{(2)} \otimes
g_{(1)} \triangleright a \\
g_{(1)} h_{(1)} \otimes f(g_{(2)}, \, h_{(2)}) &{=}& g_{(2)}
h_{(2)} \otimes f(g_{(1)}, \, h_{(1)})
\end{eqnarray*}
for all $g$, $h \in H$ and $a$, $b \in A$.

\subsection*{Bicrossed product of Hopf algebras}\selabel{1.2}
The bicrossed product of Hopf algebras was introduced by Majid in
\cite[Proposition 3.12]{majid} under the name of double cross
product. We shall adopt the name of bicrossed product from
\cite[Theorem 2.3]{Kassel}. A \textit{matched pair} of bialgebras
is a system $(A, H, \lhd, \rhd)$, where $A$ and $H$ are
bialgebras, $\triangleleft : H \otimes A \rightarrow H$,
$\triangleright: H \otimes A \rightarrow A$ are coalgebra maps
such that $(A, \rhd)$ is a left $H$-module coalgebra, $(H, \lhd)$
is a right $A$-module coalgebra and the following compatibility
conditions hold:
\begin{eqnarray}
1_{H} \triangleleft a &{=}& \varepsilon_{A}(a)1_{H}, \,\,\,
h \triangleright 1_{A}  = \varepsilon_{H}(h)1_{A} \eqlabel{mp1}\\
g \triangleright (ab) &{=}& (g_{(1)} \triangleright a_{(1)}) \bigl
( (g_{(2)}\triangleleft a_{(2)})\triangleright b \bigl)
\eqlabel{mp2} \\
(g h) \triangleleft a &{=}& \bigl( g \triangleleft (h_{(1)}
\triangleright a_{(1)}) \bigl) (h_{(2)} \triangleleft a_{(2)})
\eqlabel{mp3} \\
g_{(1)} \triangleleft a_{(1)} \otimes g_{(2)} \triangleright
a_{(2)} &{=}& g_{(2)} \triangleleft a_{(2)} \otimes g_{(1)}
\triangleright a_{(1)} \eqlabel{mp4}
\end{eqnarray}
for all $a$, $b\in A$, $g$, $h\in H$. Let $(A, H, \lhd, \rhd)$ be
a matched pair of bialgebras; the \textit{bicrossed product} $A
\bowtie H$ of $A$ with $H$ is the $k$-module $A\ot H$ with the
multiplication given by
\begin{equation}\eqlabel{0010}
(a \bowtie h) \cdot (c \bowtie g):= a (h_{(1)}\triangleright
c_{(1)}) \bowtie (h_{(2)} \lhd c_{(2)}) g
\end{equation}
for all $a$, $c\in A$, $h$, $g\in H$, where we denoted $a\ot h$ by
$a\bowtie h$. $A \bowtie H$ is a bialgebra with the coalgebra
structure given by the tensor product of coalgebras and moreover,
if $A$ and $H$ are Hopf algebras, then $A \bowtie H$ has an
antipode given by the formula:
\begin{equation}\eqlabel{antipbic}
S ( a \bowtie h ) := (1_A \bowtie S_H (h)) \cdot (S_A (a) \bowtie
1_H)
\end{equation}
for all $a\in A$ and $h\in H$ \cite[Theorem 7.2.2]{majid2}.

\section{Bialgebra extending structures and unified products} \selabel{uni1}
In this section we shall introduce the unified product for
bialgebras; this will be the tool for answering  the restricted
(H-C) ES-problem. First we need the following:

\begin{definition}\delabel{noua1}
Let $A$ be a bialgebra. An \textit{extending datum of $A$} is a
system $\Omega(A) = \bigl(H, \triangleleft, \, \triangleright, \,
f \bigl)$ where:

$(i)$ $H = \bigl( H, \Delta_{H}, \varepsilon_{H}, 1_{H}, \cdot
\bigl)$ is a $k$-module such that $\bigl( H, \Delta_{H},
\varepsilon_{H}\bigl)$ is a coalgebra, $\bigl( H, 1_{H}, \cdot
\bigl)$ is an unitary, not necessarily associative $k$-algebra,
such that
\begin{equation}\eqlabel{1}
\Delta_{H}(1_{H}) = 1_{H} \otimes
1_{H}
\end{equation}
$(ii)$ The $k$-linear maps $\triangleleft : H \otimes A
\rightarrow H$, $\triangleright: H \otimes A \rightarrow A$, $f: H
\otimes H \rightarrow A$ are morphisms of coalgebras such that the
following normalization conditions hold:
\begin{equation}\eqlabel{2}
\quad h \triangleright 1_{A} = \varepsilon_{H}(h)1_{A}, \quad
1_{H} \triangleright a = a, \quad 1_{H} \triangleleft a =
\varepsilon_{A}(a)1_{H}, \quad h\triangleleft 1_{A} = h
\end{equation}
\begin{equation}\eqlabel{3}
f(h, 1_{H}) = f(1_{H}, h) = \varepsilon_{H}(h)1_{A}
\end{equation}
for all $h \in H$, $a \in A$.
\end{definition}

Let $A$ be a bialgebra and $\Omega(A) = \bigl(H, \triangleleft, \,
\triangleright, \, f \bigl)$ an extending datum of $A$. We denote
by $A \ltimes_{\Omega(A)} H = A \ltimes H$ the $k$-module $A
\otimes H$ together with the multiplication:
\begin{equation}\eqlabel{10}
(a \ltimes h)\bullet(c \ltimes g) := a(h_{(1)}\triangleright
c_{(1)})f\bigl(h_{(2)}\triangleleft c_{(2)}, \, g_{(1)}\bigl) \,
\ltimes \, (h_{(3)}\triangleleft c_{(3)}) \cdot g_{(2)}
\end{equation}
for all $a, c \in A$ and $h, g \in H$, where we denoted $a \otimes
h \in A \otimes H$ by $a \ltimes h$.

\begin{definition}
Let $A$ be a bialgebra and $\Omega(A) = \bigl(H, \triangleleft,
\triangleright, f \bigl)$ be an extending datum of $A$. The object
$A \ltimes H$ introduced above is called \textit{the unified
product of $A$ and $\Omega(A)$} if $A \ltimes H$ is a bialgebra
with the multiplication given by \equref{10}, the unit $1_{A}
\ltimes 1_{H}$ and the coalgebra structure given by the tensor
product of coalgebras, i.e.:
\begin{eqnarray}
\Delta_{A \ltimes H} (a \ltimes h) &{=}& a_{(1)} \ltimes h_{(1)}
\otimes a_{(2)}
\ltimes h_{(2)}\eqlabel{11}\\
\varepsilon_{A \ltimes H} (a \ltimes h) &{=}&
\varepsilon_{A}(a)\varepsilon_{H}(h)\eqlabel{12}
\end{eqnarray}
for all $h \in H$, $a \in A$. In this case the extending datum
$\Omega(A) = (H, \triangleleft, \, \triangleright, \, f)$ is
called a \textit{bialgebra extending structure} of $A$. The maps
$\rhd$ and $\lhd$ are called the \textit{actions} of $\Omega(A)$
and $f$ is called the $(\rhd, \lhd)$-\textit{cocycle} of
$\Omega(A)$. A bialgebra extending structure $\Omega(A) = (H,
\triangleleft, \triangleright, f)$ is called a \textit{Hopf
algebra extending structure} of $A$ if $A \ltimes H$ has an
antipode.
\end{definition}

The multiplication given by \equref{10} has a rather complicated
formula; however, for some specific elements we obtain easier
forms which will be useful for future computations.

\begin{lemma}
Let $A$ be a bialgebra and $\Omega(A) = \bigl(H, \triangleleft, \,
\triangleright, \, f \bigl)$ an extending datum of $A$. The
following cross-relations hold:
\begin{eqnarray}
(a \ltimes 1_{H})\bullet(c \ltimes g) &{=}& ac \ltimes
g\eqlabel{13}\\
(a \ltimes g)\bullet(1_{A} \ltimes h) &{=}& af(g_{(1)}, \,
h_{(1)})
\ltimes g_{(2)}\cdot h_{(2)}\eqlabel{14}\\
(a \ltimes g)\bullet(b \ltimes 1_{H}) &{=}& a(g_{(1)}
\triangleright b_{(1)}) \ltimes g_{(2)} \triangleleft
b_{(2)}\eqlabel{15}
\end{eqnarray}
\end{lemma}
\begin{proof}
Straightforward using the normalization conditions
\equref{1}-\equref{3}.
\end{proof}

It follows from \equref{13} that the map $i_{A}: A \rightarrow A
\ltimes H$, $i_{A}(a) := a \ltimes 1_{H}$, for all $a \in A$, is a
$k$-algebra map and
\begin{equation}\eqlabel{16}
(a \ltimes 1_{H})\bullet(1_{A} \ltimes g) = a \ltimes g
\end{equation}
for all $a\in A$ and $g\in H$. Hence the set $T := \{a \ltimes
1_{H} ~|~ a \in A\} \cup \{1_{A} \ltimes g ~|~ g \in H\}$ is a
system of generators as an algebra for $A \ltimes H$ and this
observation will turn out to be essential in proving the next
theorem which provides necessary and sufficient conditions for $A
\ltimes H$ to be a bialgebra: it is the Hopf algebra version of
\cite[Theorem 2.6]{am3} where the unified product for groups is
constructed.

\begin{theorem}\thlabel{1}
Let $A$ be a bialgebra and $\Omega(A) = \bigl(H, \triangleleft, \,
\triangleright, \, f \bigl)$ an extending datum of $A$. The
following statements are equivalent:

$(1)$ $A \ltimes H$ is an unified product;

$(2)$ The following compatibilities hold:
\begin{enumerate}
\item[(2a)] $\Delta_{H} : H \to H\otimes H$ and
$\varepsilon_{H} : H \to k$ are $k$-algebra maps;\\
\item[(2b)] $(H, \lhd)$ is a right $A$-module
structure;\\
\item[(2c)] $(g\cdot h)\cdot l = \bigl(g \triangleleft
f(h_{(1)}, \, l_{(1)})\bigl)\cdot (h_{(2)}\cdot l_{(2)})$\\
\item[(2d)] $g \triangleright (ab) = (g_{(1)} \triangleright
a_{(1)})[(g_{(2)}\triangleleft a_{(2)})\triangleright b]$\\
\item[(2e)] $(g\cdot h) \triangleleft a = [g \triangleleft
(h_{(1)}
\triangleright a_{(1)})] \cdot (h_{(2)} \triangleleft a_{(2)})$\\
\item[(2f)] $[g_{(1)} \triangleright (h_{(1)} \triangleright
a_{(1)})]f\Bigl(g_{(2)} \triangleleft (h_{(2)} \triangleright
a_{(2)}), \, h_{(3)} \triangleleft a_{(3)}\Bigl) =
f(g_{(1)}, \, h_{(1)})[(g_{(2)} \cdot h_{(2)}) \triangleright a]$\\
\item[(2g)] $\Bigl(g_{(1)} \triangleright f(h_{(1)}, \,
l_{(1)})\Bigl) f\Bigl(g_{(2)} \triangleleft f(h_{(2)}, \,
l_{(2)}), \, h_{(3)} \cdot l_{(3)}\Bigl) =
f(g_{(1)}, \, h_{(1)})f(g_{(2)} \cdot h_{(2)}, \, l)$\\
\item[(2h)] $g_{(1)} \triangleleft a_{(1)} \otimes g_{(2)}
\triangleright a_{(2)} = g_{(2)} \triangleleft a_{(2)} \otimes
g_{(1)} \triangleright a_{(1)}$\\
\item[(2i)] $g_{(1)} \cdot h_{(1)} \otimes f(g_{(2)}, \, h_{(2)})
= g_{(2)} \cdot h_{(2)} \otimes f(g_{(1)}, \, h_{(1)})$
\end{enumerate}
for all $g, h, l \in H$ and $a, b \in A$.
\end{theorem}

Before going into the proof of the theorem, we have a few
observations on the relations $(2a)-(2i)$ in \thref{1}. Although
they look rather complicated at first sight, they are in fact
quite natural and can be interpreted as follows: (2a) and (2b)
show that $(H, \Delta_{H}, \varepsilon_{H}, 1_{H}, \cdot )$ is a
non-associative bialgebra and a right $A$-module coalgebra via
$\lhd$. (2c) measures how far $(H, 1_H, \cdot)$ is from being an
associative algebra. (2d), (2e) and (2h) are exactly,
mutatis-mutandis, the compatibility conditions \equref{mp1} -
\equref{mp4} appearing in the definition of a matched pair of
bialgebras. (2f) and (2g) are deformations via the action
$\triangleleft$ of the twisted module condition $(5)$ and
respectively of the cocycle condition \equref{cc} which appears in
the definition of the crossed product for Hopf algebras. (2i) is a
symmetry condition for the cocycle $f$ similar to $(2h)$. Both
relations are trivially fulfilled if, for example, $H$ is
cocommutative or $f$ is the trivial cocycle.

\begin{proof}
We prove \thref{1} in several steps. From \equref{13} and
\equref{15} it is straightforward that $1_{A} \ltimes 1_{H}$ is a
unit for the algebra $(A \ltimes H, \bullet)$. Next, we prove that
$\varepsilon_{A \ltimes H}$ given by \equref{12} is an algebra map
if and only if $\varepsilon_{H}:H \rightarrow k$ is an algebra
map. For $h, g \in H$ we have:
\begin{eqnarray*}\varepsilon_{A \ltimes H}\bigl((1_{A} \ltimes h)
\bullet (1_{A} \ltimes g)\bigl) &\stackrel{\equref{14}} {=}&
\varepsilon_{A \ltimes
H}\bigl(f(h_{(1)},g_{(1)}) \ltimes g_{(2)}\cdot h_{(2)}\bigl)\\
&\stackrel{\equref{4}} {=}&
\varepsilon_{H}(h_{(1)})\varepsilon_{H}(g_{(1)})\varepsilon_{H}(g_{(2)}\cdot
h_{(2)})\\
&{=}& \varepsilon_{H}(g \cdot h)
\end{eqnarray*}
and $\varepsilon(1_{A} \ltimes h) \varepsilon(1_{A} \ltimes g) =
\varepsilon_{H}(h) \varepsilon_{H}(g)$. Thus, if $\varepsilon_{A
\ltimes H}$ is a $k$-algebra map then $\varepsilon_{H}$ is a
$k$-algebra map. Conversely, suppose that $\varepsilon_{H}$ is a
$k$-algebra map. Then, we have:
\begin{eqnarray*}
\varepsilon_{A \ltimes H}\bigl((a \ltimes h) \bullet (c \ltimes
g)\bigl) &{=}&
\varepsilon_{A}(a)\varepsilon_{H}(h_{(1)})\varepsilon_{A}(c_{(1)})
\varepsilon_{H}(h_{(2)})\varepsilon_{A}(c_{(2)})\varepsilon_{H}
(g_{(1)})\varepsilon_{H}(h_{(3)})\\
&&\varepsilon_{A}(c_{(3)})\varepsilon_{H}(g_{(2)})\\
&{=}&\varepsilon_{A}(a)\varepsilon_{H}(h)\varepsilon_{A}(c)\varepsilon_{H}(g)\\
&{=}&\varepsilon_{A \ltimes H}(a \ltimes h) \varepsilon_{A \ltimes
H}(c \ltimes g)
\end{eqnarray*}
for all $a$, $c\in A$ and $h\in H$ i.e. $\varepsilon_{A \ltimes
H}$ is an algebra map.

The next step is to prove that $\Delta_{A \ltimes H}$ is a
$k$-algebra map if and only if $\Delta_{H}:H \rightarrow H \otimes
H$ is a $k$-algebra map and the relations $(2h)$, $(2i)$ hold.
Observe that $\Delta_{A \ltimes H} (1_{A} \ltimes 1_{H})
\stackrel{\equref{1}} {=} 1_{A} \ltimes 1_{H} \otimes 1_{A}
\ltimes 1_{H}$. Since $T = \{a \ltimes 1_{H} ~|~ a \in A\} \cup
\{1_{A} \ltimes g ~|~ g \in H\}$ generates $A \ltimes H$ as an
algebra, $\Delta_{A \ltimes H}$ is a $k$-algebra map if and only
if $\Delta_{A \ltimes H}(xy) = \Delta_{A \ltimes H}(x) \Delta_{A
\ltimes H}(y)$ for all $x, y \in T$. First, observe that:
\begin{eqnarray*}
\Delta_{A \ltimes H}\bigl((a \ltimes 1_{H}) \bullet (b \ltimes
1_{H}) \bigl)&{=}& \Delta_{A \ltimes H}(ab \ltimes 1_{H})\\
&{=}& a_{(1)}b_{(1)} \ltimes 1_{H} \otimes a_{(2)}b_{(2)} \ltimes
1_{H}\\
&{=}& \bigl(a_{(1)} \ltimes 1_{H} \otimes a_{(2)} \ltimes
1_{H}\bigl) \bigl(b_{(1)} \ltimes 1_{H} \otimes b_{(2)} \ltimes
1_{H}\bigl)\\
&{=}&\Delta_{A \ltimes H}(a \ltimes 1_{H})\Delta_{A \ltimes H}(b
\ltimes 1_{H})
\end{eqnarray*}
and
\begin{eqnarray*}
\Delta_{A \ltimes H}\bigl((a \ltimes 1_{H}) \bullet (1_{A} \ltimes
g) \bigl)&{=}& \Delta_{A \ltimes H}(a \ltimes g)\\
&{=}& a_{(1)} \ltimes g_{(1)} \otimes a_{(2)} \ltimes
g_{(2)}\\
&{=}& \bigl(a_{(1)} \ltimes 1_{H} \otimes a_{(2)} \ltimes
1_{H}\bigl) \bigl(1_{A} \ltimes g_{(1)} \otimes 1_{A} \ltimes
g_{(2)}\bigl)\\
&{=}& \Delta_{A \ltimes H}(a \ltimes 1_{H})\Delta_{A \ltimes
H}(1_{A} \ltimes g)
\end{eqnarray*}
for all $a, b \in A$, $g \in H$. There are two more relations to
consider; for $g, h \in H$ we have:
\begin{eqnarray*}
\Delta_{A \ltimes H}\bigl((1_{A} \ltimes g) \bullet (1_{A} \ltimes
h) \bigl)&\stackrel{\equref{14}} {=}&
\Delta_{A \ltimes H}\bigl(f(g_{(1)},h_{(1)}) \ltimes g_{(2)} \cdot h_{(2)}\bigl)\\
&\stackrel{\equref{4}} {=}& f(g_{(1)(1)},h_{(1)(1)}) \ltimes
(g_{(2)} \cdot h_{(2)})_{(1)} \otimes \\
&& f(g_{(1)(2)},h_{(1)(2)}) \ltimes (g_{(2)} \cdot
h_{(2)})_{(2)}\\
&{=}& f(g_{(1)},h_{(1)}) \ltimes
(g_{(3)} \cdot h_{(3)})_{(1)} \otimes \\
&& f(g_{(2)},h_{(2)}) \ltimes (g_{(3)} \cdot h_{(3)})_{(2)}
\end{eqnarray*}
and
\begin{eqnarray*}
\Delta_{A \ltimes H}(1_{A} \ltimes g) \Delta_{A \ltimes H}(1_{A}
\ltimes h)&{=}&\bigl(1_{A} \ltimes g_{(1)} \otimes 1_{A} \ltimes
g_{(2)}\bigl) \bigl(1_{A} \ltimes h_{(1)} \otimes 1_{A} \ltimes
h_{(2)}\bigl)\\
&\stackrel{\equref{14}} {=}&f(g_{(1)},h_{(1)}) \ltimes g_{(2)}
\cdot h_{(2)} \otimes f(g_{(3)},h_{(3)}) \ltimes g_{(4)} \cdot
h_{(4)}
\end{eqnarray*}
Thus $\Delta_{A \ltimes H}\bigl((1_{A} \ltimes g) \bullet (1_{A}
\ltimes h) = \Delta_{A \ltimes H}(1_{A} \ltimes g) \Delta_{A
\ltimes H}(1_{A} \ltimes h)$ if and only if
\begin{eqnarray*}\eqlabel{17}
f(g_{(1)}, \, h_{(1)}) \ltimes (g_{(3)} \cdot h_{(3)})_{(1)}
\otimes
f(g_{(2)}, \, h_{(2)}) \ltimes (g_{(3)} \cdot h_{(3)})_{(2)} = \\
 = f(g_{(1)}, h_{(1)}) \ltimes g_{(2)} \cdot h_{(2)} \otimes
f(g_{(3)}, h_{(3)}) \ltimes g_{(4)} \cdot h_{(4)}
\end{eqnarray*}
We show now that this relation holds if and only if $\Delta_{H}:H
\rightarrow H \otimes H$ is a $k$-algebra map and $(2i)$ holds.
Indeed, suppose first that the above relation holds. By applying
$\varepsilon_A \otimes Id \otimes \varepsilon_A \otimes Id$ to it
we obtain $\Delta_{H}(g \cdot h) = g_{(1)} \cdot h_{(1)} \otimes
g_{(2)} \cdot h_{(2)}$, i.e. $\Delta_{H}$ is a $k$-algebra map.
Furthermore, if we apply $\varepsilon_A \otimes Id \otimes Id
\otimes \varepsilon_H$ to it we obtain $g_{(1)} \cdot h_{(1)}
\otimes f(g_{(2)},h_{(2)}) = g_{(2)} \cdot h_{(2)} \otimes
f(g_{(1)},h_{(1)})$, i.e. $(2i)$. Conversely, suppose that
$\Delta_{H}$ is a $k$-algebra map and $(2i)$ holds. We then have:
\begin{eqnarray*}
&&f(g_{(1)}, \, h_{(1)}) \ltimes (g_{(3)} \cdot h_{(3)})_{(1)}
\otimes
f(g_{(2)}, \, h_{(2)}) \ltimes (g_{(3)} \cdot h_{(3)})_{(2)} = \\
&{=}& f(g_{(1)}, \, h_{(1)}) \ltimes g_{(3)} \cdot h_{(3)} \otimes
f(g_{(2)}, \, h_{(2)}) \ltimes g_{(4)} \cdot h_{(4)}\\
&{=}& f(g_{(1)}, \, h_{(1)}) \ltimes \underline{g_{(2)(2)} \cdot
h_{(2)(2)} \otimes f(g_{(2)(1)}, \, h_{(2)(1)})} \ltimes g_{(3)}
\cdot
h_{(3)}\\
& \stackrel{(2i)} {=}& f(g_{(1)}, \, h_{(1)}) \ltimes
g_{(2)(1)}\cdot h_{(2)(1)} \otimes f(g_{(2)(2)}, \, h_{(2)(2)})
\ltimes g_{(3)} \cdot
h_{(3)}\\
&{=}& f(g_{(1)}, \, h_{(1)}) \ltimes g_{(2)} \cdot h_{(2)} \otimes
f(g_{(3)}, \, h_{(3)}) \ltimes g_{(4)} \cdot h_{(4)}
\end{eqnarray*}
as needed. To end with, for the last family of generators we have:
\begin{eqnarray*}
\Delta_{A \ltimes H}\bigl((1_{A} \ltimes g) \bullet (a \ltimes
1_{H})\bigl) &{\stackrel{\equref{15}}=}& \Delta_{A \ltimes
H}(g_{(1)}
\triangleright a_{(1)} \ltimes g_{(2)} \triangleleft a_{(2)})\\
&{\stackrel{\equref{6},\equref{8}}=}& g_{(1)} \triangleright
a_{(1)} \ltimes g_{(3)} \triangleleft a_{(3)} \otimes g_{(2)}
\triangleright a_{(2)} \ltimes g_{(4)} \triangleleft a_{(4)}
\end{eqnarray*}
and
\begin{eqnarray*}
\Delta_{A \ltimes H}(1_{A} \ltimes g)\Delta_{A \ltimes H}(a
\ltimes 1_{H}) &{=}& \bigl(1_{A} \ltimes g_{(1)} \otimes 1_{A}
\ltimes g_{(2)}\bigl) \bigl(a_{(1)} \ltimes 1_{H} \otimes a_{(2)}
\ltimes 1_{H}\bigl)\\
&{\stackrel{\equref{15}}=}& g_{(1)} \triangleright a_{(1)} \ltimes
g_{(2)} \triangleleft a_{(2)} \otimes g_{(3)} \triangleright
a_{(3)} \ltimes g_{(4)} \triangleleft a_{(4)}
\end{eqnarray*}
Thus $\Delta_{A \ltimes H} \bigl( (1_{A} \ltimes g) \bullet (a
\ltimes 1_{H}) \bigl) = \Delta_{A \ltimes H}(1_{A} \ltimes
g)\Delta_{A \ltimes H}(a \ltimes 1_{H})$ if and only if
\begin{eqnarray*}\eqlabel{18}
g_{(1)} \triangleright a_{(1)} \ltimes g_{(3)} \triangleleft
a_{(3)} \otimes g_{(2)} \triangleright a_{(2)} \ltimes g_{(4)}
\triangleleft a_{(4)} = \\
= g_{(1)} \triangleright a_{(1)} \ltimes g_{(2)} \triangleleft
a_{(2)} \otimes g_{(3)} \triangleright a_{(3)} \ltimes g_{(4)}
\triangleleft a_{(4)}
\end{eqnarray*}
This relation is equivalent to the compatibility condition $(2h)$:
indeed, by applying $\varepsilon_A \otimes Id \otimes Id \otimes
\varepsilon_H$ to it we obtain $(2h)$. Conversely suppose that
$(2h)$ holds. Then:
\begin{eqnarray*}
&& g_{(1)} \triangleright a_{(1)} \ltimes \,\, \underline{g_{(3)}
\triangleleft a_{(3)} {\otimes} g_{(2)} \triangleright a_{(2)}}
\ltimes g_{(4)} \triangleleft a_{(4)} = \\
&& \stackrel{(2h)} {=} g_{(1)} \triangleright a_{(1)} \ltimes \,\,
g_{(2)} \triangleleft a_{(2)} \otimes g_{(3)} \triangleright
a_{(3)} \ltimes g_{(4)} \triangleleft a_{(4)}
\end{eqnarray*}
as needed.

To resume, we proved until now that $\Delta_{A \ltimes H}$ and
$\varepsilon_{A \ltimes H}$ are $k$-algebra maps if and only if
the relations $(2a)$, $(2h)$, $(2i)$ hold. In what follows we
shall prove, in the hypothesis that $\Delta_{A \ltimes H}$ and
$\varepsilon_{A \ltimes H}$ are $k$-algebra maps, that the
multiplication given by \equref{10} is associative if and only if
the compatibility conditions $(2b)$-$(2g)$ hold. This will end the
proof. We make use again of the fact that $T$ generates $A \ltimes
H$ as an algebra. Thus $\bullet$ is associative if and only if $x
\bullet (y \bullet z) = (x \bullet y) \bullet z$, for all $x$,
$y$, $z \in T$. To start with, we will prove that:
$$
(a \ltimes 1_{H}) \bullet (y \bullet z)
= [(a \ltimes 1_{H}) \bullet y] \bullet z
$$
for all $a \in A$ and $y$, $z \in T$. Indeed, we have:
\begin{eqnarray*}
(a \ltimes 1_{H}) \bullet \Bigl((1_{A} \ltimes g) \bullet (b
\ltimes 1_{H})\Bigl)&\stackrel{\equref{15}} {=}& (a \ltimes 1_{H})
\bullet (g_{(1)} \triangleright b_{(1)} \ltimes g_{(2)}
\triangleleft b_{(2)})\\
&\stackrel{\equref{13}} {=}& a(g_{(1)} \triangleright b_{(1)})
\ltimes (g_{(2)} \triangleleft b_{(2)}))\\
&\stackrel{\equref{15}} {=}& (a \ltimes g) \bullet (b \ltimes
1_{H})\\
&{=}& \Bigl((a \ltimes 1_{H}) \bullet (1_{A} \ltimes g)\Bigl)
\bullet (b \ltimes 1_{H})
\end{eqnarray*}
and
\begin{eqnarray*}
(a \ltimes 1_{H}) \bullet \Bigl((1_{A} \ltimes g) \bullet (1_{A}
\ltimes h)\Bigl)&\stackrel{\equref{14}} {=}& (a \ltimes 1_{H})
\bullet (f(g_{(1)},h_{(1)}) \ltimes g_{(2)}\cdot h_{(2)})\\
&\stackrel{\equref{13}} {=}& af(g_{(1)},h_{(1)}) \ltimes
g_{(2)}\cdot h_{(2)}\\
&\stackrel{\equref{14}} {=}& (a \ltimes g) \bullet (1_{A} \ltimes
h)\\
&\stackrel{\equref{13}} {=}&\Bigl((a \ltimes 1_{H}) \bullet (1_{A}
\ltimes g)\Bigl) \bullet (1_{A} \ltimes h)
\end{eqnarray*}
The other two possibilities for choosing the elements of $T$ can
also be proven by a straightforward computation. Thus $\bullet$ is
associative if and only if $(1_A \ltimes g) \bullet (y \bullet z)
= [(1_A \ltimes g) \bullet y] \bullet z$, for all $g \in H$, $y$,
$z \in T$. First we note that:
\begin{eqnarray*}
(1_{A} \ltimes g) \bullet \Bigl((a \ltimes 1_{H}) \bullet (1_{A}
\ltimes h)\Bigl) &{=}& (1_{A} \ltimes g) \bullet (a \ltimes h)\\
&{=}& (g_{(1)} \triangleright a_{(1)})f(g_{(2)} \triangleleft
a_{(2)}, h_{(1)}) \ltimes (g_{(3)} \triangleleft a_{(3)}) \cdot
h_{(2)}\\
&\stackrel{\equref{14}} {=}& (g_{(1)} \triangleright a_{(1)}
\ltimes g_{(2)} \triangleleft a_{(2)}) \bullet (1_{A} \ltimes h)\\
&\stackrel{\equref{15}} {=}& \Bigl((1_{A} \ltimes g) \bullet (a
\ltimes 1_{H})\Bigl) \bullet (1_{A} \ltimes h)
\end{eqnarray*}
On the other hand:
\begin{eqnarray*}
(1_{A} \ltimes g) \bullet \Bigl((b \ltimes 1_{H}) \bullet (c
\ltimes 1_{H})\Bigl) &{=}& (1_{A} \ltimes g) \bullet (bc \ltimes
1_{H})\\
&\stackrel{\equref{15}} {=}& g_{(1)} \triangleright
(b_{(1)}c_{(1)}) \ltimes g_{(2)} \triangleleft (b_{(2)}c_{(2)})
\end{eqnarray*}
and
\begin{eqnarray*}
\Bigl((1_{A} \ltimes g) \bullet (b \ltimes 1_{H})\Bigl) \bullet (c
\ltimes 1_{H})&\stackrel{\equref{15}} {=}& (g_{(1)} \triangleright
b_{(1)} \ltimes g_{(2)} \triangleleft b_{(2)}) \bullet (c \ltimes
1_{H})\\
&\stackrel{\equref{15}} {=}& (g_{(1)} \triangleright
b_{(1)})\bigl((g_{(2)} \triangleleft b_{(2)}) \triangleright
c_{(1)} \bigl) \ltimes (g_{(3)} \triangleleft b_{(3)})
\triangleleft c_{(2)}
\end{eqnarray*}
Hence $(1_{A} \ltimes g) \bullet \Bigl((b \ltimes 1_{H}) \bullet
(c \ltimes 1_{H})\Bigl) = \Bigl((1_{A} \ltimes g) \bullet (b
\ltimes 1_{H})\Bigl) \bullet (c \ltimes 1_{H})$ if and only if
\begin{equation}\eqlabel{19}
g_{(1)} \triangleright (b_{(1)}c_{(1)}) \ltimes g_{(2)}
\triangleleft (b_{(2)}c_{(2)}) = (g_{(1)} \triangleright
b_{(1)})\bigl((g_{(2)} \triangleleft b_{(2)}) \triangleright
c_{(1)} \bigl) \ltimes (g_{(3)} \triangleleft b_{(3)})
\triangleleft c_{(2)}
\end{equation}
for all $b, c \in A$ and $g \in H$. We show now that the relation
\equref{19} is equivalent to the compatibility conditions $(2b)$
and $(2d)$. Indeed, by applying $\varepsilon_A \otimes Id$ and
respectively $Id \otimes \varepsilon_H$ in \equref{19} we obtain
relations $(2b)$ respectively $(2d)$. Conversely, suppose that
relations $(2b)$ and $(2d)$ hold. We then have:
\begin{eqnarray*}
&& (g_{(1)} \triangleright b_{(1)})\bigl((g_{(2)} \triangleleft
b_{(2)}) \triangleright c_{(1)} \bigl) \ltimes (g_{(3)}
\triangleleft b_{(3)}) \triangleleft c_{(2)} = \\
& = & (g_{(1)(1)} \triangleright b_{(1)(1)})\bigl((g_{(1)(2)}
\triangleleft b_{(1)(2)}) \triangleright c_{(1)} \bigl) \ltimes
(g_{(2)} \triangleleft b_{(2)})
\triangleleft c_{(2)}\\
&\stackrel{(2d)} {=}& g_{(1)} \triangleright (b_{(1)}c_{(1)})
\otimes (g_{(2)} \triangleleft b_{(2)}) \triangleleft c_{(2)}\\
&\stackrel{(2b)} {=}& g_{(1)} \triangleright (b_{(1)}c_{(1)})
\otimes g_{(2)} \triangleleft (b_{(2)} c_{(2)})
\end{eqnarray*}
i.e. \equref{19} holds. Now we deal with the last two cases. Since
$\triangleright$ is a coalgebra map we obtain:
\begin{eqnarray*}
(1_{A} \ltimes g) \bullet \Bigl((1_{A} \ltimes h) \bullet (a
\ltimes 1_{H})\Bigl)&\stackrel{\equref{15}} {=}& (1_{A} \ltimes g)
\bullet (h_{(1)} \triangleright a_{(1)} \ltimes h_{(2)}
\triangleleft a_{(2)})\\
&\stackrel{\equref{8}} {=}& \Bigl(g_{(1)} \triangleright (h_{(1)}
\triangleright a_{(1)})\Bigl)f(g_{(2)} \triangleleft (h_{(2)}
\triangleright a_{(2)}), \, h_{(4)} \triangleleft a_{(4)})\\
&& \ltimes [g_{(3)} \triangleleft (h_{(3)} \triangleright
a_{(3)})]\cdot (h_{(5)} \triangleleft a_{(5)})
\end{eqnarray*}
and
\begin{eqnarray*}
\Bigl((1_{A} \ltimes g) \bullet (1_{A} \ltimes h)\Bigl) \bullet (a
\ltimes 1_{H})&\stackrel{\equref{14}} {=}& \bigl(f(g_{(1)}, \,
h_{(1)}) \ltimes g_{(2)} \cdot h_{(2)}\bigl)
\bullet (a \ltimes 1_{H})\\
&\stackrel{\equref{15}} {=}& f(g_{(1)}, \, h_{(1)})[(g_{(2)} \cdot
h_{(2)}) \triangleright a_{(1)}] \ltimes (g_{(3)} \cdot h_{(3)})
\triangleleft a_{(2)}
\end{eqnarray*}
Thus $(1_{A} \ltimes g) \bullet \Bigl((1_{A} \ltimes h) \bullet (a
\ltimes 1_{H})\Bigl) = \Bigl((1_{A} \ltimes g) \bullet (1_{A}
\ltimes h)\Bigl) \bullet (a \ltimes 1_{H})$ if and only if
\begin{eqnarray*}\eqlabel{20}
&\Bigl(g_{(1)} \triangleright (h_{(1)} \triangleright
a_{(1)})\Bigl)f \Bigl ( g_{(2)} \triangleleft (h_{(2)}
\triangleright a_{(2)}), \, h_{(4)} \triangleleft a_{(4)} \Bigl)
\ltimes [g_{(3)} \triangleleft (h_{(3)} \triangleright
a_{(3)})]\cdot (h_{(5)} \triangleleft a_{(5)}) \\
& =  f(g_{(1)}, \, h_{(1)})[(g_{(2)} \cdot h_{(2)}) \triangleright
a_{(1)}] \ltimes (g_{(3)} \cdot h_{(3)}) \triangleleft a_{(2)}
\end{eqnarray*}
We shall prove, using $(2h)$, that this relation is equivalent to
the compatibility conditions $(2e)$ and $(2f)$. Indeed, by
applying $Id \otimes \varepsilon_H$ and respectively
$\varepsilon_A \otimes Id$ to it we obtain $(2e)$ and respectively
$(2f)$. Conversely, suppose that relations $(2e)$ respectively
$(2f)$ hold. We denote ${\rm LHS}$ the left hand side of the above
relation. We have:
\begin{eqnarray*}
{\rm LHS} &{=}& \bigl(g_{(1)} \triangleright (h_{(1)}
\triangleright a_{(1)})\bigl) f \bigl(g_{(2)} \triangleleft
(h_{(2)} \triangleright
a_{(2)}), \, h_{(3)(2)} \triangleleft a_{(3)(2)} \bigl) \ltimes\\
&&[g_{(3)} \triangleleft (h_{(3)(1)} \triangleright
a_{(3)(1)})]\cdot
(h_{(4)} \triangleleft a_{(4)})\\
&\stackrel{(2h)} {=}& \Bigl(\underline{g_{(1)} \triangleright
(h_{(1)} \triangleright a_{(1)})\Bigl)f(g_{(2)} \triangleleft
(h_{(2)} \triangleright a_{(2)}), \, h_{(3)} \triangleleft
a_{(3)}})
\ltimes \\
&&[g_{(3)} \triangleleft (h_{(4)} \triangleright a_{(4)})]\cdot
(h_{(5)} \triangleleft a_{(5)})\\
&\stackrel{(2f)} {=}& f(g_{(1)}, \, h_{(1)})[(g_{(2)} \cdot
h_{(2)}) \triangleright a_{(1)}] \otimes [g_{(3)} \triangleleft
(h_{(3)}
\triangleright a_{(2)})]\cdot (h_{(4)} \triangleleft a_{(3)})\\
&\stackrel{(2e)} {=} & f(g_{(1)}, \, h_{(1)})[(g_{(2)} \cdot
h_{(2)}) \triangleright a_{(1)}] \otimes (g_{(3)} \cdot h_{(3)})
\triangleleft a_{(2)}
\end{eqnarray*}
as needed. Only one associativity relation remains to be verified:
\begin{eqnarray*}
(1_{A} \ltimes g) \bullet \Bigl((1_{A} \ltimes h) \bullet (1_{A}
\ltimes l)\Bigl) &\stackrel{\equref{14}} {=}& (1_{A} \ltimes g)
\bullet (f(h_{(1)}, \, l_{(1)}) \ltimes h_{(2)} \cdot l_{(2)})\\
&\stackrel{\equref{4}} {=}& \bigl(g_{(1)} \triangleright
f(h_{(1)}, l_{(1)})\bigl)f\bigl(g_{(2)} \triangleleft
f(h_{(2)}, l_{(2)}), \, h_{(4)} \cdot l_{(4)}\bigl) \\
&&\ltimes \bigl(g_{(3)} \triangleleft f(h_{(3)}, \, l_{(3)})\bigl)
\cdot (h_{(5)} \cdot l_{(5)})
\end{eqnarray*}
and
\begin{eqnarray*}
\Bigl((1_{A} \ltimes g) \bullet (1_{A} \ltimes h)\Bigl) \bullet
(1_{A} \ltimes l) &\stackrel{\equref{14}} {=}& \bigl(f(g_{(1)}, \,
h_{(1)}) \ltimes g_{(2)} \cdot h_{(2)}\bigl)
\bullet (1_{A} \ltimes l)\\
&\stackrel{\equref{14}} {=}& f(g_{(1)}, \, h_{(1)})f\bigl(g_{(2)}
\cdot h_{(2)}, \, l_{(1)}\bigl) \ltimes (g_{(3)} \cdot h_{(3)})
\cdot l_{(2)}
\end{eqnarray*}
Hence $(1_{A} \ltimes g) \bullet \Bigl((1_{A} \ltimes h) \bullet
(1_{A} \ltimes l)\Bigl) = \Bigl((1_{A} \ltimes g) \bullet (1_{A}
\ltimes h)\Bigl) \bullet (1_{A} \ltimes l)$ if and only if
\begin{eqnarray*}\eqlabel{21}
& \bigl(g_{(1)} \triangleright
f(h_{(1)},l_{(1)})\bigl)f\Bigl(g_{(2)} \triangleleft
f(h_{(2)},l_{(2)}),h_{(4)} \cdot l_{(4)}\Bigl) \otimes
\bigl(g_{(3)} \triangleleft f(h_{(3)},l_{(3)})\bigl) \cdot
(h_{(5)} \cdot l_{(5)}) = \\
& = f(g_{(1)},h_{(1)})f\bigl(g_{(2)} \cdot h_{(2)},l_{(1)}\bigl)
\otimes (g_{(3)} \cdot h_{(3)}) \cdot l_{(2)}
\end{eqnarray*}
for all $g, h, l \in H$. We shall prove, using $(2i)$, that this
relation is equivalent to the compatibility conditions $(2c)$ and
$(2g)$. Indeed, by applying $Id \otimes \varepsilon_H$ and
respectively $\varepsilon_A \otimes Id$ to it we obtain $(2c)$ and
respectively $(2g)$. Conversely, suppose that relations $(2c)$ and
$(2g)$ hold and denote ${\rm LHS'}$ the left hand side of the
above relation. Then:
\begin{eqnarray*}
{\rm LHS'} &{=}& \bigl(g_{(1)} \triangleright f(h_{(1)}, \,
l_{(1)})\bigl) f\Bigl(g_{(2)} \triangleleft
f(h_{(2)}, \, l_{(2)}), \underline{h_{(3)(2)} \cdot l_{(3)(2)}}\Bigl) \otimes\\
&&\bigl(g_{(3)} \triangleleft \underline{f(h_{(3)(1)}, \,
l_{(3)(1)})}\bigl) \cdot (h_{(4)} \cdot
l_{(4)})\\
&\stackrel{(2i)} {=}& \bigl(\underline{g_{(1)} \triangleright
f(h_{(1)}, \, l_{(1)})\bigl) f\Bigl(g_{(2)} \triangleleft
f(h_{(2)}, \, l_{(2)}), \, h_{(3)} \cdot l_{(3)}}\Bigl) \\
&&\otimes \bigl(g_{(3)} \triangleleft f(h_{(4)}, \, l_{(4)})\bigl)
\cdot (h_{(5)} \cdot l_{(5)})\\
&\stackrel{(2g)} {=}& f(g_{(1)}, \, h_{(1)})f(g_{(2)} \cdot
h_{(2)},l_{(1)}) \otimes \bigl(g_{(3)} \triangleleft
f(h_{(3)}, \, l_{(2)})\bigl) \cdot (h_{(4)} \cdot l_{(3)})\\
&{=}& f(g_{(1)}, \, h_{(1)})f(g_{(2)} \cdot h_{(2)}, \, l_{(1)})
\otimes \underline{\bigl(g_{(3)} \triangleleft f(h_{(3)(1)}, \,
l_{(2)(1)})\bigl) \cdot (h_{(3)(2)} \cdot
l_{(2)(2)})}\\
&\stackrel{(2c)} {=}& f(g_{(1)}, \, h_{(1)})f(g_{(2)} \cdot
h_{(2)}, \, l_{(1)}) \otimes (g_{(3)} \cdot h_{(3)}) \cdot l_{(2)}
\end{eqnarray*}
as needed and the proof is now finished.
\end{proof}

\begin{examples}\exlabel{3exemple}
1. Let $A$ be a bialgebra and $\Omega(A) = \bigl(H, \triangleleft,
\triangleright, f \bigl)$ an extending datum of $A$ such that the
cocycle $f$ is trivial, that is $f (g, \, h) = \varepsilon_H (g)
\varepsilon_H (h) 1_A$, for all $g$, $h\in H$.

Then $\Omega(A) = \bigl(H, \triangleleft, \triangleright, f
\bigl)$ is a bialgebra extending structure of $A$ if and only if
$H$ is a bialgebra and $(A, H, \triangleleft, \triangleright)$ is
a matched pair of bialgebras. In this case, the associated unified
product $A\ltimes H = A\bowtie H$ is the bicrossed product of
bialgebras constructed in \equref{0010}.

Conversely, a matched pair of bialgebras can be deformed using a
coalgebra lazy cocycle in order to obtain a bialgebra extending
structure as follows. Let $(A, H, \triangleleft, \triangleright)$
be a matched pair of bialgebras such that $A$ has antipode $S_A$
and $u: H \to A$ a coalgebra lazy $1$-cocycle in the sense of
\deref{lazydef} such that $h \triangleleft u(g) = h
\varepsilon_H(g)$, for all $h\in H$ and $g\in G$. Then $\Omega(A)
= \bigl(H, \triangleleft, \triangleright', f' \bigl)$ is a
bialgebra extending structure of $A$, where $\triangleright'$ and
$f'$ are given by
\begin{eqnarray*}
h\triangleright' c &{=}& u(h_{(1)}) (h_{(2)} \triangleright
c_{(1)}) S_{A}\Bigl(u \bigl(h_{(3)} \triangleleft
c_{(2)}\bigl)\Bigl)  \eqlabel{C2} \\
f'(h , \, g) &{=}& u(h_{(1)})(h_{(2)} \triangleright u(g_{(1)}))
S_{A}\Bigl(u\bigl(h_{(3)} g_{(2)}\bigl)\Bigl)
\end{eqnarray*}
for all $h$, $g\in H$ and $c\in A$.

2. Let $A$ be a bialgebra and $\Omega(A) = \bigl(H, \triangleleft,
\, \triangleright, \, f \bigl)$ an extending datum of $A$ such
that the action $\lhd$ is trivial, that is $h \lhd a =
\varepsilon_A (a) h$, for all $h\in H$ and $a\in A$.

Then $\Omega(A) = \bigl(H, \triangleleft, \, \triangleright, \, f
\bigl)$ is a bialgebra extending structure of $A$ if and only if
$H$ is an usual bialgebra and the following compatibility
conditions are fulfilled:
\begin{enumerate}
\item[(a)] The twisted module condition $(5)$ and the cocycle
condition \equref{cc} hold;

\item[(b)] $g \triangleright (ab) = (g_{(1)} \triangleright a)
(g_{(2)} \triangleright b)$

\item[(c)] $g_{(1)} \otimes g_{(2)} \triangleright a = g_{(2)}
\otimes g_{(1)} \triangleright a$

\item[(d)] $g_{(1)} h_{(1)} \otimes f(g_{(2)}, \, h_{(2)}) =
g_{(2)} h_{(2)} \otimes f(g_{(1)}, \, h_{(1)})$
\end{enumerate}
for all $g$, $h \in H$ and $a$, $b \in A$.

In this case, the associated unified product $A\ltimes H = A
\#_{f} \, H$ is the crossed product constructed in \equref{001}.
In particular, if $A$ is a bialgebra, the crossed product $A
\#_{f} \, H$ is a bialgebra with the coalgebra structure given by
the tensor product of coalgebras if and only if the compatibility
conditions (c) and (d) above hold.
\end{examples}

Let $A$ be a bialgebra and $\Omega(A) = \bigl(H, \triangleleft, \,
\triangleright, \, f \bigl)$ a bialgebra extending structure of
$A$. Then $i_A : A \to A  \ltimes H$,  $i_A (a) = a  \ltimes 1_H$,
for all $a\in A$ is an injective bialgebra map, $i_H : H \to
A\ltimes H$, $i_H (h) = 1_A \ot h$, for all $h\in H$ is an
injective coalgebra map and
$$
u: A \otimes H \rightarrow A \ltimes H, \qquad u(a \otimes h) =
i_A(a) \bullet i_H(h) = (a \ltimes 1_{H}) \bullet (1_{A} \ltimes
h) = a \ltimes h
$$
for all $a\in A$ and $h\in H$ is bijective, i.e. the unified
product $A \ltimes H$ factorizes through $A$ and $H$. The next
theorem shows the converse of this remark: any bialgebra $E$ that
factorizes through a subbialgebra of $A$ and a subcoalgebra $H$ is
isomorphic to a unified product. In order to avoid complicated
computations we use the following elementary remark:

\begin{lemma}\lelabel{lema}
Let $E$ be a bialgebra, $L$ a coalgebra and $u: L \rightarrow E$
an  isomorphism of coalgebras. Then there exists a unique algebra
structure on $L$ such that $u: L \rightarrow E$ is an isomorphism
of bialgebras given by:
$$
l \cdot l' := u^{-1}(u(l)u(l')) , \qquad 1_{L} := u^{-1}(1_{E})
$$
for all $l$, $l' \in L$. Furthermore, if $E$ has an antipode
$S_E$, then $L$ is a Hopf algebra with the antipode $S_L := u^{-1}
\circ S_E \circ u$.
\end{lemma}

\begin{proof} Straightforward: the algebra structure on $L$ is
obtained by transfering the algebra structure from $E$ via the
isomorphism of coalgebras $u$. The multiplication on $L$ is a
coalgebra map since it is a composition of coalgebra maps.
\end{proof}

\begin{theorem}\thlabel{2}
Let $E$ be a bialgebra, $A \subseteq E$ a subbialgebra, $H
\subseteq E$ a subcoalgebra such that $1_{E} \in H$ and the
multiplication map $u: A\otimes H \rightarrow E$, $u(a \otimes h)
= ah $, for all $a\in A$, $h\in H$ is bijective.

Then, there exists $\Omega(A)=(H, \triangleleft, \triangleright,
f)$ a bialgebra extending structure of $A$ such that $u: A \ltimes
H \rightarrow E$, $u(a \ltimes h) = ah$ is an isomorphism of
bialgebras. Furthermore, if $E$ is a Hopf algebra then $A \ltimes
H$ is a Hopf algebra.
\end{theorem}

\begin{proof}
Since $E$ is a bialgebra, the multiplication $m_{E}: E \otimes E
\rightarrow E$ is a coalgebra map. Thus $u:A \otimes H \rightarrow
E$ is in fact an isomorphism of coalgebras, with its inverse
$u^{-1}: E \rightarrow A \otimes H$ which is also a coalgebra map.
The $k$-linear map
$$
\mu : H\otimes A \to A\ot H, \quad \mu (h\ot a) := u^{-1} (ha)
$$
for all $h\in H$ and $a\in A$ is a coalgebra map as a composition
of coalgebra maps. We define the actions $\triangleright$,
$\triangleleft$ by the formulas:
\begin{eqnarray}
\triangleright: H \otimes A \rightarrow A , \qquad \triangleright
&:=& (Id \otimes \varepsilon_{H}) \circ \mu \eqlabel{prima}\\
\triangleleft: H \otimes A \rightarrow H , \qquad \triangleleft
&:=& (\varepsilon_{A} \otimes Id) \circ \mu \eqlabel{adoua}
\end{eqnarray}
They are coalgebra maps as compositions of coalgebra maps.
Moreover, the normalization conditions \equref{2} and \equref{3}
are trivially fulfilled. More explicitly, $\triangleright$ and
$\triangleleft$ are given as follows: let $h \in H$ and $c \in A$.
Since $u$ is a bijective map, there exists an unique element
$\sum_{j} \alpha_{j} \otimes l_{j} \in A \otimes H$ such that $hc
= \sum_{j} \alpha_{j} l_{j}$. Then:
$$
h \triangleright c = \sum_{j} \alpha_{j} \varepsilon_{H}(l_{j})  ,
\qquad  h \triangleleft c = \sum_{j}
\varepsilon_{A}(\alpha_{j})l_{j}
$$
Next we construct the coalgebra maps $f: H \otimes H \rightarrow
A$ and $\cdot : H \otimes H \rightarrow H$. The $k$-linear map
$$
\nu : H\ot H \to A\ot H, \quad \nu (h\ot g) := u^{-1} (hg)
$$
for all $h$, $g \in H$ is a coalgebra map as a composition of
coalgebra maps. We define:
\begin{eqnarray}
f: H \otimes H \rightarrow A , \qquad f &:=& (Id \otimes
\varepsilon_{H}) \circ \nu \eqlabel{22} \\
\cdot : H \otimes H \rightarrow H , \qquad \cdot &:=&
(\varepsilon_{A} \otimes Id) \circ \nu \eqlabel{23}
\end{eqnarray}
They are coalgebra maps as compositions of coalgebra maps. The
normalization conditions $1_{E} \cdot h = h \cdot 1_{E} = h$ and
$f(h, 1_{E}) = f(1_{E}, h) = \varepsilon_{H}(h)1_{A}$, for all $h
\in H$ are trivially fulfilled.

In order to prove that $\Omega(A) = (H, \triangleleft, \,
\triangleright, \, f, \, \cdot)$ is a bialgebra extending
structure of $A$ we use \leref{lema} and then \thref{1}: the
unique algebra structure that can be defined on $A \otimes H$ such
that $u$ becomes an isomorphism of bialgebras is given by:
\begin{eqnarray*}
(a \otimes h) \bullet (c \otimes g) &{=}& u^{-1}\bigl(u(a \otimes
h) u(c \otimes g)\bigl)\\
&{=}& u^{-1} (ahcg)
\end{eqnarray*}
This algebra structure on $A \otimes H$ coincides with the one
given by \equref{10} on a unified product if and only if
$$
u^{-1}(ahcg) = a (h_{(1)}
\triangleright c_{(1)})f(h_{(2)} \triangleleft c_{(2)},g_{(1)})
\otimes (h_{(3)} \triangleleft c_{(3)}) \cdot g_{(2)}
$$
Since $u$ is a bijective map the above formula holds if and only
if:
\begin{equation}\eqlabel{24}
hcg = (h_{(1)} \triangleright c_{(1)})f(h_{(2)} \triangleleft
c_{(2)},g_{(1)})\bigl((h_{(3)} \triangleleft c_{(3)}) \cdot
g_{(2)}\bigl)
\end{equation}
holds for all $c \in A$ and $h, g \in H$. Therefore, the proof is
finished if we prove that the relation \equref{24} holds in the
bialgebra $E$. Let $c \in A$ and $h$, $g \in H$. Then there exists
an unique element $\sum_{j=1}^{n} \alpha_{j} \otimes l_{j} \in A
\otimes H$ such that:
\begin{equation}\eqlabel{25}
hc = \sum_{j=1}^{n} \alpha_{j} l_{j}
\end{equation}
Hence $h \triangleright c = \sum_{j=1}^{n} \varepsilon_{H}(l_{j})
\alpha_{j}$ and $h \triangleleft c = \sum_{j=1}^{n}
\varepsilon_{A}(\alpha_{j})l_{j}$. Moreover, for any $j = 1,
\cdots, n$ there exists an unique element $\sum_{i=1}^{m} A_{ji}
\otimes Z_{i} \in A \otimes H$ such that:
\begin{equation}\eqlabel{26}
l_{j}g = \sum_{i=1}^{m} A_{ji} Z_{i}
\end{equation}
Using relations \equref{22} and \equref{23} we obtain:
\begin{equation}\eqlabel{27}
f(l_{j},g) = \sum_{i=1}^{m} \varepsilon_{H}(Z_{i})A_{ji} , \qquad
l_{j} \cdot g = \sum_{i=1}^{m} \varepsilon_{A}(A_{ji})Z_{i}
\end{equation}
and
\begin{equation}\eqlabel{28}
hcg = \sum_{i,j=1}^{m,n} \alpha_{j} A_{ji}Z_{i}
\end{equation}
In what follows we use the fact that $m_{E}$, $\triangleright$ and
$\triangleleft$ are coalgebra maps. For example, by applying
$\Delta$ to the relation \equref{25} we obtain:
\begin{eqnarray*}
h_{(1)} \triangleright c_{(1)} \otimes h_{(2)} \triangleleft
c_{(2)} \otimes h_{(3)} \triangleleft c_{(3)} &{=}& \sum_{j=1}^{n}
\varepsilon_{H}(l_{j(1)}) \alpha_{j(1)} \otimes
\varepsilon_{A}(\alpha_{j(2)})l_{j(2)} \otimes
\varepsilon_{A}(\alpha_{j(3)})l_{j(3)}\\
&{=}& \sum_{j=1}^{n} \varepsilon_{H}(l_{j(1)}) \alpha_{j} \otimes
l_{j(2)} \otimes l_{j(3)}\\
&{=}& \sum_{j=1}^{n}  \alpha_{j} \otimes l_{j(1)} \otimes l_{j(2)}
\end{eqnarray*}
Thus, we have:
\begin{equation}\eqlabel{29}
h_{(1)} \triangleright c_{(1)} \otimes h_{(2)} \triangleleft
c_{(2)} \otimes h_{(3)} \triangleleft c_{(3)} = \sum_{j=1}^{n}
\alpha_{j} \otimes l_{j(1)} \otimes l_{j(2)}
\end{equation}
Moreover, by applying $\Delta$ to the relation \equref{26} and
using the relation \equref{27} we obtain:
\begin{equation}\eqlabel{334}
f\bigl(l_{j(1)}, \, g_{(1)}\bigl) \otimes l_{j(1)} \cdot g_{(2)} =
\sum_{i=1}^{m} \varepsilon_{H}(Z_{i_{(1)}}) A_{ji_{(1)}} \otimes
\varepsilon_{A}(A_{ji_{(2)}}) Z_{i_{(2)}} = \sum_{i=1}^m A_{ji}
\otimes Z_{i}
\end{equation}
We denote by ${\rm RHS}$ the right hand side of \equref{24}. Then:
\begin{eqnarray*}
{\rm RHS} &\stackrel{\equref{29}} {=}& \sum_{j=1}^{n} \alpha_{j}
f\bigl(l_{j(1)}, \, g_{(1)}\bigl) \, l_{j(2)} \cdot g_{(2)}\\
&\stackrel{\equref{334}} {=}&
\sum_{i,j=1}^{m,n} \alpha_{j} A_{ji}Z_{i} \\
&\stackrel{\equref{28}} {=}& hcg
\end{eqnarray*}
Thus the relation \equref{24} holds true and the proof is now
finished since $u^{-1}(1_{E}) = 1_{A} \otimes 1_{E}$. We use
\thref{1} in order to obtain that $\Omega(A) = (H, \triangleleft,
\, \triangleright, \, f, \, \cdot)$ is a bialgebra extending
structure of $A$. Moreover, if $E$ is a Hopf algebra then $A
\ltimes H$ is also a Hopf algebra with the antipode given by $S_{A
\ltimes H} = u^{-1} \circ S_{E} \circ u$ according to
\leref{lema}.
\end{proof}

Next we construct an antipode for the unified product $A \ltimes
H$.

\begin{proposition}\prlabel{prext}
Let $A$ be a Hopf algebra with an antipode $S_A$ and $\Omega(A) =
(H, \triangleleft, \triangleright, f)$ a bialgebra extending
structure of $A$ such that there exists an antimorphism of
coalgebras $S_{H}: H \rightarrow H$ such that
\begin{equation}\eqlabel{a-bial3}
h_{(1)} \cdot S_{H}(h_{(2)}) = S_{H}(h_{(1)}) \cdot h_{(2)} =
\varepsilon_{H}(h)1_{H}
\end{equation}
for all $h \in H$. Then the unified product $A \ltimes H$ is a
Hopf algebra with the antipode $S : A \ltimes H \rightarrow A
\ltimes H$ given by:
\begin{equation}\eqlabel{antipod}
S(a \ltimes g) := \Bigl(S_{A}[f\bigl(S_{H}(g_{(2)}), \,
g_{(3)}\bigl)] \ltimes S_{H}(g_{(1)})\Bigl) \bullet \bigl(S_{A}(a)
\ltimes 1_{H}\bigl)
\end{equation}
for all $a\in A$ and $g\in H$.
\end{proposition}

\begin{proof}
Let $a \ltimes g \in A \ltimes H$. Since the multiplication
$\bullet$ on $A \ltimes H$ is associative we have:
\begin{eqnarray*}
&&S(a_{(1)}\ltimes g_{(1)}) \bullet (a_{(2)} \ltimes g_{(2)}) = \\
&=& \Bigl(S_{A}[f\bigl(S_{H}(g_{(2)}),g_{(3)}\bigl)] \ltimes
S_{H}(g_{(1)})\Bigl) \bullet \bigl(S_{A}(a_{(1)}) \ltimes
1_{H}\bigl) \bullet (a_{(2)} \ltimes g_{(4)})\\
&\stackrel{\equref{13}} {=}&
\Bigl(S_{A}[f\bigl(S_{H}(g_{(2)}),g_{(3)}\bigl)] \ltimes
S_{H}(g_{(1)})\Bigl) \bullet \bigl(S_{A}(a_{(1)})a_{(2)} \ltimes
g_{(4)}\bigl)\\
&{=}& \varepsilon_{A}(a)\Bigl(S_{A}[f\bigl(S_{H}(g_{(2)}),
g_{(3)}\bigl)]
\ltimes S_{H}(g_{(1)})\Bigl) \bullet (1_{A} \ltimes g_{(4)})\\
&\stackrel{\equref{14}} {=}&
\varepsilon_{A}(a)S_{A}\Bigl(f\bigl(S_{H}(g_{(2)}),
g_{(3)}\bigl)\Bigl) f\bigl(S_{H}(g_{(1)})_{(1)},g_{(4)(1)}\bigl)
\ltimes S_{H}(g_{(1)})_{(2)} \cdot g_{(4)(2)}\\
&{=}& \varepsilon_{A}(a)S_{A}\Bigl(f\bigl(S_{H}(g_{(1)(2)}),
g_{(2)}\bigl)\Bigl)
f\bigl(S_{H}(g_{(1)(1)})_{(1)},g_{(3)(1)}\bigl)
\ltimes S_{H}(g_{(1)(1)})_{(2)} \cdot g_{(3)(2)}\\
&{=}& \varepsilon_{A}(a)S_{A}\Bigl(f\bigl(S_{H}(g_{(1)})_{(1)},
g_{(2)}\bigl)\Bigl)
f\bigl(S_{H}(g_{(1)})_{(2)(1)},g_{(3)(1)}\bigl)
\ltimes S_{H}(g_{(1)})_{(2)(2)} \cdot g_{(3)(2)}\\
&{=}& \varepsilon_{A}(a)S_{A}\Bigl(f\bigl(S_{H}(g_{(1)})_{(1)},
g_{(2)}\bigl)\Bigl) f\bigl(S_{H}(g_{(1)})_{(2)},g_{(3)(1)}\bigl)
\ltimes S_{H}(g_{(1)})_{(3)} \cdot g_{(3)(2)}\\
&{=}& \varepsilon_{A}(a)S_{A}\Bigl(f\bigl(S_{H}(g_{(1)})_{(1)},
g_{(2)(1)}\bigl)\Bigl)
f\bigl(S_{H}(g_{(1)})_{(2)},g_{(2)(2)(1)}\bigl)
\ltimes S_{H}(g_{(1)})_{(3)} \cdot g_{(2)(2)(2)}\\
&{=}& \varepsilon_{A}(a)S_{A}\Bigl(f\bigl(S_{H}(g_{(1)})_{(1)},
g_{(2)(1)}\bigl)\Bigl)
f\bigl(S_{H}(g_{(1)})_{(2)},g_{(2)(2)}\bigl)
\ltimes S_{H}(g_{(1)})_{(3)} \cdot g_{(2)(3)}\\
&{=}& \varepsilon_{A}(a)S_{A}\Bigl(f\bigl(S_{H}(g_{(1)})_{(1)(1)},
g_{(2)(1)(1)}\bigl)\Bigl)
f\bigl(S_{H}(g_{(1)})_{(1)(2)},g_{(2)(1)(2)}\bigl)
\ltimes\\
&&S_{H}(g_{(1)})_{(2)} \cdot g_{(2)(2)}\\
&\stackrel{\equref{4}} {=}&
\varepsilon_{A}(a)S_{A}\Bigl(f\bigl(S_{H}(g_{(1)})_{(1)},g_{(2)(1)}\bigl)_{(1)}\Bigl)
f\bigl(S_{H}(g_{(1)})_{(1)},g_{(2)(1)}\bigl)_{(2)}
\ltimes S_{H}(g_{(1)})_{(2)} \cdot g_{(2)(2)}\\
&{=}& \varepsilon_{A}(a)\varepsilon_{A}
\Bigl(f\bigl(S_{H}(g_{(1)})_{(1)},g_{(2)(1)}\bigl)\Bigl)1_{A}
\ltimes S_{H}(g_{(1)})_{(2)} \cdot g_{(2)(2)}\\
&\stackrel{\equref{4}} {=}&
\varepsilon_{A}(a)\varepsilon_{H}\Bigl(S_{H}(g_{(1)})_{(1)}\Bigl)
\varepsilon_{H}\bigl(g_{(2)(1)}\bigl)\ltimes S_{H}(g_{(1)})_{(2)}
\cdot g_{(2)(2)}\\
&{=}&\varepsilon_{A}(a)1_{A}\ltimes S_{H}(g_{(1)}) \cdot g_{(2)}\\
&{=}&\varepsilon_{A}(a)\varepsilon_{H}(g)1_{A} \ltimes 1_{H}
\end{eqnarray*}
Thus $S$ is a left inverse of the identity in the convolution
algebra ${\rm Hom} ( A \ltimes H, A \ltimes H)$. By similar
computations one can show that $S$ is also a right inverse of the
identity, thus is an antipode of $A\ltimes H$.
\end{proof}

In \prref{prext} we imposed the condition for $S_{H}$ to be a
coalgebra antimorphism because the algebra structure on $H$ is not
an associative one and for this reason a $k$-linear map $S_{H}$
which satisfies the antipode condition \equref{a-bial3} is not
necessarily a coalgebra antimorphism as in the classical case of
Hopf algebras.

\section{The classification of unified products}\selabel{sec3}

In this section we prove the classification theorem for unified
products: as a special case, a classification theorem for
bicrossed products of Hopf algebras is obtained. Our view point is
inspired from Schreier's classification theorem for extensions of
an abelian group $K$ by a group $Q$ \cite[Theorem 7.34]{R}: they
are classified by the second cohomology group $H^2 (Q, K)$. Let
$\varphi : G \to G'$ be a morphism between two extensions of a
group $K$ by a group $Q$, i.e. $\varphi$ is a morphism of groups
such that the diagram
\begin{eqnarray*}
\xymatrix {& k[K] \ar[r]^{i} \ar[d]_{Id} & {k[G]}
\ar[r]^{\pi}\ar[d]^{\varphi} & k[Q] \ar[d]^{Id}\\
& k[K]\ar[r]^{i'} & {k[G']}\ar[r]^{\pi '} & k[Q]}
\end{eqnarray*}
is commutative (we wrote the diagram for the induced morphism for
group algebras). Then $\varphi$ is an isomorphism \cite[Theorem
7.32]{R}. Now, the left hand square of the diagram is commutative
if and only if $\varphi$ is a left $k[K]$-module map while the
right hand square of the diagram is commutative if and only if
$\varphi$ is a morphism of right $k[Q]$-comodules. This motivates
the way of considering the classification of unified products up
to an isomorphism of Hopf algebras that is also a left $A$-module
map and a right $H$-comodule map.

Let $\Omega(A) = \bigl(H, \triangleleft, \, \triangleright, \, f
\bigl)$ be a bialgebra extending structure of $A$. The unified
product $A \ltimes H$ is a right $H$-comodule via the coaction $a
\ltimes h \mapsto a \ltimes h_{(1)} \ot h_{(2)}$, for all $a \in
A$ and $h\in H$ and a left $A$-module via the restriction of
scalars map $i_A : A \to A\ltimes H$.

From now on the Hopf algebra structure on $A$ and the coalgebra
structure on $H$ will be set. First, we need the following.

\begin{lemma}\lelabel{3.1}
Let $A$ be a Hopf algebra, $\Omega(A) = \bigl(H, \triangleleft, \,
\triangleright, \, f \bigl)$ and $\Omega'(A) = \bigl(H,
\triangleleft', \, \triangleright', \, f' \bigl)$ two Hopf algebra
extending structures of $A$. Then a $k$-linear map $\varphi: A
\ltimes H \rightarrow A \ltimes' H $ is a left $A$-module, a right
$H$-comodule and a coalgebra morphism if and only if there exists
a unique morphism of coalgebras $u: H\to A$ such that
\begin{equation}\eqlabel{3.0}
h_{(1)} \otimes u(h_{(2)}) = h_{(2)} \otimes u(h_{(1)})
\end{equation}
for all $h\in H$ and $\varphi$ is given by
\begin{equation}\eqlabel{3.0.01}
\varphi(a \ltimes h) = au(h_{(1)}) \ltimes' h_{(2)}
\end{equation}
for all $a\in A$ and $h\in H$. Furthermore, any such a morphism
$\varphi: A \ltimes H \rightarrow A \ltimes' H $ is an isomorphism
with the inverse given by
$$
\psi: A \ltimes' H \rightarrow A \ltimes H, \quad \psi(a \ltimes'
h) = aS_{A}\bigl(u(h_{(1)})\bigl) \ltimes h_{(2)}
$$
for all $a\in A$ and $h\in H$.
\end{lemma}

\begin{proof}
Let $\varphi: A \ltimes H \rightarrow A \ltimes' H $ be a left
$A$-module, a right $H$-comodule and a coalgebra morphism. We
shall adopt the notation $\varphi (1_A \ltimes h) = \sum h^A \ot
h^H \in A\ot H$, for all $h\in H$. Since $\varphi$ is a left
$A$-module map we have
$$
\varphi(a \ltimes h) = a \varphi(1_A \ltimes h) = a \sum h^A \ot
h^H
$$
for all $a\in A$ and $h\in H$. As $\varphi$ is also a right
$H$-comodule map we have:
$$
\sum a h^A \ot (h^H)_{(1)} \ot (h^H)_{(2)} = \varphi(a \ltimes
h_{(1)} ) \ot h_{(2)}
$$
By applying $\varepsilon_H$ on the second position of the above
identity we obtain:
$$
\varphi(a \ltimes h) = \sum a (h_{(1)})^A \varepsilon_H (
(h_{(1)})^H ) \ot h_{(2)}
$$
for all $a\in A$ and $h\in H$. Now, if we define $u: H \rightarrow
A$ by:
$$
u(h) = (Id \otimes \varepsilon_H) \circ \varphi (1_A \ltimes h) =
\sum h^A \varepsilon_H (h^H)
$$
for all $h\in H$, it follows that \equref{3.0.01} holds. We shall
prove now that $\varphi$ given by \equref{3.0.01} is a coalgebra
map if and only if $u$ is a coalgebra map and \equref{3.0} holds.
First we observe that $\varepsilon_{A \ltimes' H } \circ \varphi =
\varepsilon_{A \ltimes H}$ if and only if $\varepsilon_A \circ u =
\varepsilon_H$. Now, the fact that $\varphi$ is comultiplicative
is equivalent to:
\begin{equation}\eqlabel{3.3}
u(h_{(1)})_{(1)} \otimes h_{(2)} \otimes u(h_{(1)})_{(2)} \otimes
h_{(3)} = u(h_{(1)}) \otimes h_{(2)} \otimes u(h_{(3)}) \otimes
h_{(4)}
\end{equation}
for all $h\in H$. By applying $Id \otimes \varepsilon_H \otimes Id
\otimes \varepsilon_H $ to this relation we obtain that $u$ is a
coalgebra map; using this fact and then applying $\varepsilon_A
\otimes Id \otimes Id \otimes \varepsilon_H $ in relation
\equref{3.3} we obtain relation \equref{3.0}. Conversely, if $u$
is a coalgebra map such that relation \equref{3.0} holds, then
\equref{3.3} follows straightforward, i.e. $\varphi$ is a
coalgebra map. The fact that $\psi$ is an inverse for $\phi$ is
also straightforward.
\end{proof}

\begin{remark}
At this point we should remark the perfect similarity with the
theory of extensions from the groups case. If $\varphi: A \ltimes
H \rightarrow A \ltimes' H $ is a left $A$-module, a right
$H$-comodule and a coalgebra morphism between two unified products
then the following diagram
\begin{eqnarray*}
\xymatrix {& A \ar[r]^{i_{A}} \ar[d]_{Id_{A}} & {A\bowtie H}
\ar[r]^{\pi_{H}}\ar[d]^{\varphi} & H\ar[d]^{Id_{H}}\\
& A\ar[r]^{i_{A}} & {A\bowtie' H}\ar[r]^{\pi_{H}} & H}
\end{eqnarray*} is commutative and $\varphi$ is an isomorphism.
\end{remark}

For $A = k$ and a Hopf algebra $H$ the group $H^1_l(H, k)$ of all
unitary algebra maps $u: H\rightarrow k$ satisfying the
compatibility condition \equref{laz} below was called in
\cite[Definition 1.1]{BK} the \emph{first lazy cohomology group}
of $H$ with coefficients in $k$. We shall now define the coalgebra
version of lazy $1$-cocyles.

\begin{definition}\delabel{lazydef}
Let $A$ be a Hopf algebra and $H$ a coalgebra, unitary not
necessarily associative algebra. A morphism of coalgebras $u: H
\rightarrow A$ is called a \emph{coalgebra lazy $1$-cocyle} if
$u(1_H) = 1_A$ and the following compatibility holds:
\begin{equation}\eqlabel{laz}
h_{(1)} \otimes u(h_{(2)})  = h_{(2)} \otimes u(h_{(1)})
\end{equation}
for all $h\in H$. We denote by $H^{1}_{l, c} (H, A)$ the group of
all coalgebra lazy $1$-cocyles of $H$ with coefficients in $A$.
\end{definition}

$H^{1}_{l, c} (H, A)$ is a group with respect to the convolution
product. We have to prove that if $u$ and $v\in H^{1}_{l, c} (H,
A)$, then $u
* v \in H^{1}_{l, c} (H, A)$. Indeed, is straightforward to prove
that $u
* v$ satisfy \equref{laz}. Let us show that $u * v$ is a morphism of
coalgebras. First, if we apply $v$ on the first position in
\equref{laz} we obtain $ v (h_{(1)}) \otimes u(h_{(2)})  = v
(h_{(2)}) \otimes u(h_{(1)})$, for all $h\in H$. Using this
relation we obtain:
\begin{eqnarray*}
\Delta_A \bigl ( u(h_{(1)}) v(h_{(2)}) \bigl) &{=}& u(h_{(1)})
v(h_{(3)}) \ot u(h_{(2)}) v(h_{(4)}) \\
&{=}& u(h_{(1)}) v(h_{(2)}) \ot u(h_{(3)}) v(h_{(4)})\\
&{=}& u * v (h_{(1)}) \ot u* v (h_{(2)})
\end{eqnarray*}
for all $h\in H$, hence $u * v$ is also a coalgebra map.

The main theorem of this section now follows:

\begin{theorem}\thlabel{3.4}
Let $A$ be a Hopf algebra, $\Omega(A) = \bigl(H, \triangleleft, \,
\triangleright, \, f \bigl)$ and $\Omega'(A) = \bigl(H,
\triangleleft', \, \triangleright', \, f' \bigl)$ two Hopf algebra
extending structures of $A$. Then there exists $\varphi: A
\ltimes' H \rightarrow A \ltimes H $ a left $A$-module, a right
$H$-comodule and a Hopf algebra map if and only if $\triangleleft'
= \triangleleft $ and there exists a coalgebra lazy $1$-cocyle $u
\in H^{1}_{l, c} (H, A)$ such that:
\begin{eqnarray}
h\triangleright' c &{=}& u(h_{(1)}) (h_{(2)} \triangleright
c_{(1)}) S_{A}\Bigl(u \bigl(h_{(3)} \triangleleft
c_{(2)}\bigl)\Bigl)  \eqlabel{C2} \\
f'(h , \, g) &{=}&
u(h_{(1)})(h_{(2)} \triangleright u(g_{(1)}))f(h_{(3)}
\triangleleft u(g_{(2)}) , \, g_{(3)})
S_{A}\Bigl(u\bigl(h_{(4)} \cdot' g_{(4)}\bigl)\Bigl) \\
\eqlabel{C3}
 h \cdot' g &{=}& \bigl(h \triangleleft u(g_{(1)})\bigl) \cdot
\, g_{(2)} \eqlabel{C4}
\end{eqnarray}
for all $h$, $g \in H$ and $c \in A$. In this case $\varphi$ is
given by \equref{3.0.01} and it is an isomorphism.
\end{theorem}

\begin{proof}
We already proved in \leref{3.1} that $\varphi: A \ltimes' H
\rightarrow A \ltimes H $ is a left $A$-module, a right
$H$-comodule and a coalgebra map if and only if $\varphi (a
\ltimes' h) = a u(h_{(1)}) \ltimes h_{(2)}$, for all $a\in A$,
$h\in H$ and for a unique coalgebra map $u: H \to A $ such that
the \equref{laz} holds. Of course, $\varphi (1_A \ltimes' 1_H) =
1_A \ltimes 1_H$ if and only if $u$ is unitary. Moreover, as $u$
is a morphism of coalgebras it is invertible in convolution with
the inverse $u^{-1} = S_A \circ u$.

In what follows we shall prove, in the hypothesis that $\varphi$
is a coalgebra map and $u$ is unitary, that $\varphi$ is an
algebra map (thus a map of bialgebras) if and only if $
\triangleleft' =  \triangleleft $ and the compatibility conditions
\equref{C2} - \equref{C4} hold. By a straightforward computation
we can show that $\varphi$ is an algebra map if and only if
\begin{eqnarray*}
(C) \,\, \, (h_{(1)} \triangleright' c_{(1)}) f'\bigl(h_{(2)}
\triangleleft' c_{(2)} , \, g_{(1)}\bigl) u\bigl( (h_{(3)}
\triangleleft' c_{(3)}) \cdot' g_{(2)}\bigl) \ltimes (h_{(4)}
\triangleleft'
c_{(4)}) \cdot' g_{(3)} = \\
= u(h_{(1)})\bigl(h_{(2)} \triangleright c_{(1)}u(g_{(1)})\bigl)
f\bigl(h_{(3)} \triangleleft c_{(2)}u(g_{(2)}) , \, g_{(4)}\bigl)
\ltimes \bigl(h_{(4)} \triangleleft c_{(3)}u(g_{(3)})\bigl) \cdot
g_{(5)}\eqlabel{3.5}
\end{eqnarray*}
for all $h$, $g \in H$ and $c \in A$. We shall prove that the
compatibility $(C)$ is equivalent to \equref{C2} - \equref{C4}.

Indeed, by considering $g = 1_H$ in $(C)$ and then by applying
$\varepsilon_A \otimes Id$ we obtain $h \triangleleft' c = h
\triangleleft c$, for all $h\in H$ and $c\in A$. If we consider
again $g = 1_H$ we obtain, after applying first $Id \otimes
\varepsilon_H$ and then inverting $u$, that \equref{C2} holds.
Relation \equref{C3} is obtained by considering $c = 1_A$ in
$(C)$, applying $Id \otimes \varepsilon_H$ and finally inverting
$u$. To end with, relation \equref{C4} follows by considering $c =
1_A$ and by applying $\varepsilon_A \otimes Id$ in $(C)$.

Conversely, suppose that $h \triangleleft' c = h \triangleleft c$,
for all $h\in H$ and $c\in A$ and there exists a coalgebra lazy
$1$-cocyle $u$ such that relations \equref{C2} - \equref{C4} are
fulfilled. We then have (we denote by $LHS$ the left hand side of
$(C)$):
\begin{eqnarray*}
LHS &\stackrel{ \equref{C2}-\equref{C3} } {=}& u(h_{(1)(1)})
(h_{(1)(2)} \triangleright c_{(1)(1)})S_{A}\Bigl(u
\bigl(h_{(1)(3)} \triangleleft
c_{(1)(2)}\bigl)\Bigl)u(h_{(2)(1)} \triangleleft c_{(2)(1)})\\
&&\bigl( (h_{(2)(2)} \triangleleft c_{(2)(2)}) \triangleright
u(g_{(1)(1)})\bigl) f\bigl( (h_{(2)(3)} \triangleleft c_{(2)(3)})
\triangleleft u(g_{(1)(2)}) , \, g_{(1)(3)}\bigl)\\
&& S_{A}\Bigl(u \bigl(h_{(2)(4)} \triangleleft c_{(2)(4)}\bigl)
\cdot' g_{(1)(4)}\Bigl) u\bigl((h_{(3)} \triangleleft c_{(3)})
\cdot' g_{(2)}\bigl) \ltimes (h_{(4)} \triangleleft c_{(4)})
\cdot' g_{(3)}\\
&{=}& u(h_{(1)}) (h_{(2)} \triangleright
c_{(1)})\underline{S_{A}\Bigl(u \bigl(h_{(3)} \triangleleft
c_{(2)}\bigl)\Bigl)u(h_{(4)} \triangleleft c_{(3)})}\bigl(
(h_{(5)} \triangleleft c_{(4)})
\triangleright u(g_{(1)})\bigl)\\
&& f\bigl( (h_{(6)} \triangleleft c_{(5)})\triangleleft u(g_{(2)})
, \, g_{(3)}\bigl) \underline{S_{A}\Bigl(u \bigl(h_{(7)}
\triangleleft c_{(6)}\bigl) \cdot' g_{(4)}\Bigl) u\bigl((h_{(8)}
\triangleleft
c_{(7)}) \cdot' g_{(5)}\bigl)}\\
&&\ltimes (h_{(9)} \triangleleft c_{(8)}) \cdot' g_{(6)}\\
&{=}& u(h_{(1)}) (h_{(2)} \triangleright c_{(1)}) \bigl( (h_{(3)}
\triangleleft c_{(2)}) \triangleright u(g_{(1)})\bigl)f\bigl(
(h_{(4)} \triangleleft c_{(3)})\triangleleft u(g_{(2)}) , \,
g_{(3)}\bigl)\\
&& \ltimes (h_{(5)} \triangleleft c_{(4)}) \cdot' g_{(4)}\\
&\stackrel{(2d)}{=}& u(h_{(1)}) \Bigl(h_{(2)} \triangleright
\bigl(c_{(1)} u(g_{(1)})\bigl)\Bigl)f\bigl( (h_{(3)} \triangleleft
c_{(2)})\triangleleft u(g_{(2)}) , \, g_{(3)}\bigl) \\
&& \ltimes (h_{(4)} \triangleleft c_{(3)}) \cdot' g_{(4)}\\
&\stackrel{\equref{C4} }{=}& u(h_{(1)}) \Bigl(h_{(2)}
\triangleright \bigl(c_{(1)} u(g_{(1)})\bigl)\Bigl)f\bigl(
(h_{(3)} \triangleleft
c_{(2)})\triangleleft u(g_{(2)}) , \, g_{(3)}\bigl)\\
&& \ltimes (h_{(4)} \triangleleft c_{(3)}u(g_{(4)})) \cdot g_{(5)}\\
&\stackrel{\equref{laz}}{=}& u(h_{(1)}) \Bigl(h_{(2)}
\triangleright \bigl(c_{(1)} u(g_{(1)})\bigl)\Bigl)f\bigl(
(h_{(3)} \triangleleft
c_{(2)})\triangleleft u(g_{(2)}) , \, g_{(4)}\bigl)\\
&& \ltimes (h_{(4)} \triangleleft c_{(3)}u(g_{(3)})) \cdot g_{(5)}\\
\end{eqnarray*}
where the third equality holds by using the antipode conditions
and the fact that $u$ is a coalgebra map. Thus $(C)$ holds and the
proof is now finished.
\end{proof}

Even if for the classification problem we only set the Hopf
algebra structure of $A$ and the coalgebra structure of $H$,
\thref{3.4} tells us that we can set also the coalgebra map $
\triangleleft : H\otimes A \to H$. We shall phrase \thref{3.4} as
a description of the skeleton for the category $\Cc (A, H,
\triangleleft )$ defined below.

Let $A$ be a Hopf algebra, $H$ a coalgebra with a fixed group-like
element $1_H \in H$ and $ \triangleleft : H\otimes A \to H$ a
morphism of coalgebras. Let ${\mathcal ES} (A, H, \triangleleft)$
be the set of all triples $(\cdot, \triangleright, \, f )$ such
that $\bigl((H, 1_X, \cdot), \triangleleft, \, \triangleright, \,
f \bigl)$ is a Hopf algebra extending structure of $A$. The next
definition is the Hopf algebra version for unified product
\cite[Definition 7.31]{R} given for extensions of groups.

\begin{definition}\delabel{equiv} Two elements
$(\cdot, \triangleright, \, f )$, $(\cdot', \triangleright', \,
f')$ of ${\mathcal ES} (A, H, \triangleleft)$ are called
\emph{cohomologous} and we denote this by $(\cdot, \triangleright,
\, f ) \approx (\cdot', \triangleright', \, f')$ if there exists a
coalgebra lazy $1$-cocyle $u \in H^{1}_{l, c} (H, A)$ such that
the compatibility conditions \equref{C2} - \equref{C4} are
fulfilled.
\end{definition}

It follows from \thref{3.4} that $(\cdot, \triangleright, \, f )
\approx (\cdot', \triangleright', \, f')$ if and only if there
exists $\varphi: A \ltimes' H \rightarrow A \ltimes H $ a left
$A$-module, a right $H$-comodule and a Hopf algebra map. Moreover,
from \leref{3.1} we obtain that any such map $\varphi: A \ltimes'
H \rightarrow A \ltimes H $ is an isomorphism, thus $\approx$ is
an equivalence relation on the set ${\mathcal ES} (A, H,
\triangleleft)$. We denote by $H^{2}_{l, c} (H, A, \triangleleft)$
the quotient set ${\mathcal ES} (A, H, \triangleleft)/\approx$.

Let $\Cc (A, H, \triangleleft )$ be the category whose class of
objects is the set ${\mathcal ES} (A, H, \triangleleft)$. A
morphism $\varphi : \bigl(\, \cdot, \triangleright, \, f \bigl)
\to \bigl(\, \cdot, \triangleright', \, f' \bigl)$ in $\Cc (A, H,
\triangleleft)$ is a Hopf algebra morphism $\varphi: A \ltimes H
\rightarrow A \ltimes' H $ that is a left $A$-module and a right
$H$-comodule map. Thus we obtain the categorical version of
\thref{3.4}:

\begin{corollary}\textbf{(Schreier theorem for unified
products)}\colabel{3.67} Let $A$ be a Hopf algebra, $H$ a
coalgebra with a group-like element $1_H$ and $ \triangleleft :
H\otimes A \to H$ a morphism of coalgebras. There exists a
bijection between the set of objects of the skeleton of the
category $\Cc (A, H, \triangleleft)$ and the quotient set
$H^{2}_{l, c} (H, A, \triangleleft)$.
\end{corollary}

$H^{2}_{l, c} (H, A, \triangleleft)$ is for the classification of
the unified products the counterpart of the second cohomology
group for the classification of an extension of an abelian group
by a group \cite[Theorem 7.34]{R}.

We can apply \thref{3.4} to obtain classification theorems for
various special cases of the unified products: for instance, Doi's
results on the classification of crossed products (\cite{D}) is
obtain as a special case if we let $ \triangleleft' =
\triangleleft $ be the trivial actions. Now, we shall indicate the
classification of bicrossed product of Hopf algebras.

\begin{corollary}\textbf{(Schreier theorem for bicrossed
products)}\colabel{3.6} Let $A$ and $H$ be two Hopf algebras and
$\bigl(A, H, \triangleleft, \, \triangleright \bigl)$, $\bigl(A,
H, \triangleleft', \, \triangleright' \bigl)$ two matched pairs of
Hopf algebras. Then $A\bowtie H \cong A\bowtie' H$ (isomorphism of
Hopf algebras, left $A$-modules and right $H$-comodules) if and
only if $\triangleleft'  =  \triangleleft$ and there exists a
coalgebra lazy $1$-cocyle $u \in H^{1}_{l, c} (H, A)$  such that:
\begin{eqnarray*}
h \triangleright' c &{=}& u(h_{(1)}) (h_{(2)} \triangleright
c_{(1)}) S_{A}\Bigl(u \bigl(h_{(3)} \triangleleft
c_{(2)}\bigl)\Bigl)\\
u(h_{(1)})(h_{(2)} \triangleright u(g_{(1)}))
S_{A}\Bigl(u\bigl(h_{(3)}  g_{(2)}\bigl)\Bigl) &{=}&
\varepsilon_{H}(g)\varepsilon_{H}(h)1_{A}\\
h \triangleleft u(g) &{=}& h \, \varepsilon_H (g)
\end{eqnarray*}
for all $h$, $g \in H$ and $c \in A$.
\end{corollary}

\begin{proof} We apply \thref{3.4} for the case when $f$ and $f'$
are the trivial cocycles. As the multiplication on the algebra $H$
is the same (i.e. $\cdot = \cdot'$), the condition \equref{C4} in
\thref{3.4} takes the equivalent form $h \triangleleft u(g) = h
\varepsilon_H (g)$, for all $h$, $g \in H$.
\end{proof}

The construction of unified products is a challenging problem
considering the number of compatibilities that need to be
fulfilled. In particular, an example of an unified product which
is neither a crossed product nor a bicrossed product is
interesting in the picture. We provide such an example below: it
is a Hopf algebra $k[A_4] \ltimes k[S] \cong k[A_6]$, where $A_n$
is the alternating group on a set with $n$ elements and $S$ is a
set with thirty elements.

\begin{example}
Let $G$ be a group and $(X, 1_X)$ a pointed set. We consider the
group Hopf algebra $A := k[G]$ and the group-like coalgebra $H :=
k[X]$. We note that coalgebra morphisms between two group-like
coalgebras are in one to one correspondence with the maps between
the corresponding sets. Thus, any bialgebra extending structure
$(k[X], \, \lhd, \, \rhd, \, f)$ of the Hopf algebra $k[G]$ is
induced by an extending structure $(X, \lhd', \, \rhd', \, f')$ of
the group $G$ in the sense of \cite[Definition 2.3]{am3}. Moreover
there exists a canonical isomorphism of bialgebras
$$
k[G] \ltimes k[X] \cong k [G \ltimes X]
$$
where $G\ltimes X$ is the unified product at the level of groups
(see \cite{am3} for further details). This generalizes the fact
that a bicrossed product of two group Hopf algebras is isomorphic
to the group Hopf algebra of the bicrossed product of the
corresponding groups \cite[Example 1, pg. 207]{Kassel}. The same
type of isomorphism holds also for crossed products of Hopf
algebras between two group Hopf algebras.

Now, let $A_6$ be the alternating group on a set with six
elements. $A_6$ is the simple group of smallest order that cannot
be written as a bicrossed product of two proper subgroups
(\cite{WW}). Being a simple group it can not be written neither as
a crossed product of two proper subgroups. On the other hand,
$A_6$ can be written as an unified product between any of its
subgroups and an extending structure. For instance, we can write
$$
A_6 \cong A_4 \ltimes S
$$
for an extending structure $\bigl( S, \, 1_{S}, \, \alpha, \,
\beta, \, f, \, \ast , i \bigl )$ of $A_4$, where $S$ is a set of
representatives for the right cosets of $A_4$ in $A_6$ with $30$
elements such that $1 \in S$. Thus there exists an example of an
unified product for Hopf algebras $k[A_4] \ltimes k[S] \cong k
[A_6]$ which is neither a crossed product nor a bicrossed product
of two Hopf algebras.
\end{example}

Two general methods for constructing unified products are proposed
in \cite{am4}. One of them constructs an unified product starting
with a minimal set of data: a Hopf algebra $A$, a unitary not
necessarily associative bialgebra $H$ which is a right $A$-module
coalgebra and a unitary coalgebra map $\gamma : H \to A$
satisfying four technical compatibility conditions (\cite[Theorem
2.9]{am4}).

\section{Acknowledgment}
A.L. Agore is ''Aspirant'' Fellow of the Fund for Scientific
Research–-Flanders (Belgium) (F.W.O.– Vlaanderen). G. Militaru was
supported from CNCSIS grant 24/28.09.07 of PN II "Groups, quantum
groups, corings and representation theory".

\end{document}